\documentclass[12pt, twoside]{article}
\usepackage{mathrsfs}
\usepackage{amsmath}
\usepackage{amssymb,amsthm,upref,amscd}
\usepackage{setspace}
\usepackage{enumerate}
\usepackage[titletoc]{appendix}
\usepackage{times}
\usepackage{cite}
\usepackage[colorinlistoftodos]{todonotes}
\numberwithin{equation}{section}

\def\titlerunning#1{\gdef\titrun{#1}}
\makeatletter
\def\author#1{\gdef\autrun{\def\and{\unskip, }#1}\gdef\@author{#1}}

\makeatother

\def\subjclass#1{{\renewcommand{\thefootnote}{}%
\footnote{\emph{Mathematics Subject Classification (2010):} #1}}}
\def\keywords#1{\par\medskip
\noindent\textbf{Keywords.} #1}

\theoremstyle{plain}
\newtheorem{Thm}{Theorem}[section]
\newtheorem{Lem}[Thm]{Lemma}
\newtheorem{claim}{Claim}[section]
\newtheorem{Cor}[Thm]{Corollary}
\newtheorem{Prop}[Thm]{Proposition}
\newtheorem*{Thm*}{Theorem}
\newtheorem*{claim*}{Claim}
\theoremstyle{definition}

\newtheorem*{Def*}{Definition}
\newtheorem*{Cor*}{Corollary}
\newtheorem{Rem}[Thm]{Remark}

\newcommand{\equ}{equation}
\newcommand{\C}{\mathbb{C}}
\newcommand{\N}{\mathbb{N}}
\newcommand{\R}{\mathbb{R}}

 \DeclareMathOperator{\dist}{dist}
 
\DeclareMathOperator{\meas}{meas}

\newcommand\eps{\varepsilon}

\newcommand{\weakto}{\rightharpoonup}

\let\nhatoksa=\theenumi
\let\nhatoksb=\labelenumi
\let\nhatoksc=\theenumii
\let\nhatoksd=\labelenumii
\newlength{\nhalengtha}
\newlength{\nhalengthb}
\newlength{\nhalengthc}
\newcommand{\resetenum}{
\let\theenumi=\nhatoksa
\let\labelenumi=\nhatoksb
\let\theenumii=\nhatoksc
\let\labelenumii=\nhatoksd
\setlength{\leftmargini}{\nhalengtha}
\setlength{\leftmarginii}{\nhalengthb}
\setlength{\labelwidth}{\nhalengthc}
}

\newcommand\cb{\mathcal{B}}
\newcommand\cc{\mathcal{C}}
\newcommand\cd{\mathcal{D}}

\newcommand\cm{\mathcal{M}}

\def\msj{\mathscr{J}}

\def\msl{\mathscr{L}}

\def\mst{\mathscr{T}}

\def\ov{\overline}

\def\pa {\partial}

\def\op{\oplus}

\def\ka {\kappa}
\def\al {\alpha}
\def\be {\beta}
\def\de {\delta}
\def\De{\Delta}
\def\ga {\gamma}
\def\Ga {\Gamma}
\def\la {\lambda}
\def\Lam {\Lambda}

\def\om {\omega}

\def\si {\sigma}
\def\eps {\varepsilon}
\def\va {\varphi}

\def\meas{\hbox{\it meas}}

\def\real{\hbox{Re}}

\def\vecmu{{\vec{\mu}}}
\newcommand{\inp}[2]{\left\langle#1,#2\right\rangle}

\begin{document}

\baselineskip=17pt

\titlerunning{Strongly localized semiclassical states for nonlinear Dirac equations}

\title{Strongly localized semiclassical states for nonlinear Dirac equations}

\author{Thomas Bartsch,\, Tian Xu\footnote{Supported by the National Science Foundation of China (NSFC 11601370, 11771325) and the Alexander von Humboldt Foundation of Germany}}

%\date{}

\maketitle


\begin{abstract}
We study semiclassical states of the nonlinear Dirac equation
\[
  -i\hbar\pa_t\psi = ic\hbar\sum_{k=1}^3\al_k\pa_k\psi - mc^2\be \psi - M(x)\psi + f(|\psi|)\psi,\quad t\in\R,\ x\in\R^3,
\]
where $V$ is a bounded continuous potential function and the nonlinear term $f(|\psi|)\psi$ is superlinear, possibly of critical growth.
Our main result deals with standing wave solutions that concentrate near a critical point of the potential. Standard methods applicable to nonlinear Schr\"odinger equations, like Lyapunov-Schmidt reduction or penalization, do not work, not even for the homogeneous nonlinearity $f(s)=s^p$. We develop a variational method for the strongly indefinite functional associated to the problem.

\end{abstract}


\section{Introduction}
Standing wave solutions for the nonlinear Schr\"odinger equation
\[
  -i\hbar\pa_t\psi = -\De \psi + V(x)\psi + f(|\psi|)\psi
\]
a non-relativistic wave equation, have been in the focus of nonlinear analysis since decades. In particular, semiclassical states that concentrate near a critical point of the potential $V$ have been widely investigated ever since the influential paper \cite{Floer} by Floer and Weinstein who treated the cubic nonlinearity $|\psi|^2\psi$ in one-dimension.

Much less is known for the nonlinear Dirac equation
\[
  -i\hbar\pa_t\psi = ic\hbar\sum_{k=1}^3\al_k\pa_k\psi - mc^2\be \psi - M(x)\psi + f(x,|\psi|)\psi,\qquad t\in\R,\ x\in\R^3,
\]
a relativistic wave equation and a spinor generalization of the nonlinear Schr\"odinger equation, not even in the case of $f$ being a pure power with subcritical nonlinearity. Here $\psi(t,x)\in\C^4$, $c$ is the speed of light, $\hbar$ is Planck's constant, $m$ is the mass of the particle and $\al_1$, $\al_2$, $\al_3$ and $\be$ are the $4\times4$ complex Pauli matrices:
\[
\be=
  \begin{pmatrix}
    I & 0 \\
    0 & -I
  \end{pmatrix}, \quad
\al_k=
  \begin{pmatrix}
    0 & \si_k \\
   \si_k & 0
  \end{pmatrix}, \quad k=1,2,3,
\]
with
\[
\si_1=
  \begin{pmatrix}
    0 & 1 \\
    1 & 0
  \end{pmatrix}, \quad
\si_2=
  \begin{pmatrix}
    0 & -i \\
    i & 0
  \end{pmatrix}, \quad
\si_3=
  \begin{pmatrix}
    1 & 0 \\
    0 & -1
  \end{pmatrix} \,.
\]
The external field $M(x)$ represents an arbitrary electric potential depending only upon $x\in\R^3$. The nonlinear coupling $f(x,|\psi|)\psi$ describes a self-interaction. Typical examples for nonlinear couplings can be found in the self-interacting scalar theories; see \cite{FLR,FFK,Iva} and more recently \cite{CV, ES, ES2, FZ, FZ2, NP, THIN}. Usually, in Quantum electrodynamics nonlinear Dirac equations have to satisfy symmetry constraints, in particular the Poincar\'e covariance. Nonlinear Dirac equations modeling Bose-Einstein condensates break this symmetry, and often the nonlinearity is a {\it power-type} function that depends only on the local condensate density (see \cite{HC, HWC, HC2} for more background from physics).

The ansatz $\psi(t,x)=e^{i\om t/\hbar}u(x)$ for a standing wave solution and a change of notation (in particular $\eps$ instead of $\hbar$) leads to an equation of the form
\begin{\equ}\label{nd2}
  -i\eps\sum_{k=1}^3\al_k\pa_ku + a\be u + V(x)u = f(x,|u|)u , \quad u\in H^1(\R^3,\C^4).
\end{\equ}
This type of particle-like solution does not change its shape as it evolves in time, hence has a soliton-like behavior.
In this paper we investigate the existence of semiclassical states, i.e.\ solutions $u_\eps$ of \eqref{nd2} for small $\eps>0$, that concentrate as $\eps\to0$ at a critical point $x_0$ of the potential $V$.
There are many results of this type for nonlinear Schr\"odinger equations
\begin{\equ}\label{Schrodinger}
  -\eps^2\De u + V(x)u = g(u), \quad u\in H^1(\R^N),
\end{\equ}
beginning with the pioneering work by Floer and Weinstein \cite{Floer} and then continued by Oh \cite{O1, O2} and many others, e.g.\ \cite{Ambrosetti1, Ambrosetti2, Bartsch-Clapp-Weth:2007, Byeon-JJ, Byeon-Wang, DPR, Del-Pino, Del-Pino2, Del-Pino3, Rab}. It has been proved that there exists a family of semiclassical solutions to \eqref{Schrodinger} for small $\eps$ which concentrate around stable critical points of the potential $V$ as $\eps\to0$. The proofs are based on Lyapunov-Schmidt type methods, penalization, and variational techniques.

Very few results are available for the nonlinear Dirac equation \eqref{nd2} compared with the nonlinear Schr\"odinger equation.  A major difference between nonlinear Schr\"odinger and Dirac equations is that the Dirac operator is strongly indefinite in the sense that both the negative and positive parts of the spectrum are unbounded and consist of essential spectrum. It follows that the quadratic part of the energy functional associated to \eqref{nd2} has no longer a positive sign, moreover, the Morse index and co-index at any critical point of the energy functional are infinite.

In order to compare our result with the existing literature we first present in short the state of the art. The first result for concentration behavior of the nonlinear Dirac equation \eqref{nd2} is due to Ding \cite{Ding2010}, who considered the case $V\equiv0$ and $f(x,|u|)=P(x)|u|^{p-2}$ with $p\in(2,3)$ subcritical, $\inf P>0$, and $\limsup_{|x|\to\infty} P(x)<\max P$. He obtained a least energy solution $u_\eps$ for $\eps>0$ small that concentrates around a global maximum of $P$ as $\eps\to0$. A similar result has been obtained in \cite{Ding-Lee-Ruf} where $f=f(|u|)$ is subcritical and $V$ satisfies $a<\min V<\liminf_{|x|\to\infty} V(x)\le|V|_\infty<a$. Here the solutions $u_\eps$ concentrate at a global minimum of $V$. In both papers \cite{Ding2010,Ding-Lee-Ruf} the solutions are obtained via a mountain pass argument applied to a reduced functional. In \cite{Ding-Xu-ARMA,WZ} the authors considered the case of a local minimum of $V$ using a penalization approach analogous to the one in \cite{Del-Pino2,Del-Pino3}.

All papers mentioned so far consider a subcritical nonlinearity $f$. The only papers dealing with a critical nonlinearity, i.e. where $f(t)$ grows as $t$ for $t\to\infty$, are \cite{Ding-Ruf,Ding-Liu-Wei}. Both papers assume, in addition to various technical conditions, that $V$ has a global minimum. The least energy solution is obtained again via a mountain pass argument applied to a reduced functional. It is essential that the mountain pass level is below the threshold level where the Palais-Smale condition fails. In \cite {Ding-Liu-Wei} the authors were also able to obtain solutions with energy above the mountain pass level using the oddness of the equation and Lusternik-Schnirelmann type arguments, but again the energy levels of the solutions are below the level where the Palais-Smale condition fails.

The distinct new feature of our result is that we find solutions of
\begin{\equ}\label{nd3}
  -i\eps\sum_{k=1}^3\al_k\pa_ku + a\be u + V(x)u = f(|u|)u
\end{\equ}
localized near a critical point of $V$ that is not necessarily a (local or global) minimum of $V$. The model nonlinearity we consider is $f(t)=\ka t +\la t^{p-2}$ with $\ka,\la>0$ and $p\in(2,3)$. We can deal with local minima, local maxima, or saddle points of $V$, both in the critical ($\ka>0$) and subcritical ($\ka=0$) case. As a consequence, a least energy solution may not exist, and in the variational setting there is no threshold value below which the Palais-Smale condition holds, so that the methods from \cite{Ding-Ruf,Ding-Liu-Wei} do not apply. We have to work at energy levels where the Palais-Smale condition fails which, in the critical case $\ka>0$, leads to a subtle interplay between $\ka,V,\la,p$. Our results are new even in the subcritical case where so far only local minima of $V$ have been treated. They are of course new in the critical case where only global minima of $V$ have been considered. %It is worthwhile to mention that we present a unified approach for the different types of critical points of $V$.

The paper is organized as follows. In the next section we state and discuss our main theorem. After collecting some basic results on the Dirac operator in Section~\ref{sec:Preliminaries} we investigate the family of equations
\begin{\equ}\label{eq:limit}
  -i\sum_{k=1}^3\al_k\pa_ku + a\be u + V(\xi)u = f(|u|)u
\end{\equ}
parametrized by $\xi\in\R^3$ which appear as limit equations. This will be done in Section~\ref{sec:limit}. In Section~\ref{sec:truncation} we introduce a truncated problem, set up the variational structure, and prove the Palais-Smale condition for the truncated functional in a certain parameter range. Then in Section~\ref{sec:min-max scheme} we develop a min-max scheme that can be applied to the truncated problem. The proof of a key result, Proposition \ref{key prop}, that is needed for the passage to the limit $\eps\to0$ will be presented in Section~\ref{sec:Proof of key prop}. The delicate analysis in Section~\ref{sec:Proof of key prop} is not needed in the case of a local minimum of $V$ because in that case  the lower bound estimate of Proposition \ref{key prop} is automatically satisfied. In the final Section~\ref{sec:Profile} we show that the solutions of the truncated problem are actually solutions of \eqref{nd2} for $\eps>0$ small enough, thus finishing the proof of the main theorem. The proof of a technical lemma will be presented in the Appendix.

\section{The main result}

We set $\al\cdot\nabla:=\sum_{k=1}^{3}\al_k\pa_k$ so that equation \eqref{nd3} reads as
\[
  -i\eps\al\cdot\nabla u+ a\be u + V(x)u = f(|u|)u , \quad u\in H^1(\R^3,\C^4).
\]
Throughout the paper, we fix the constant $a>0$ and assume that the potential $V$ satisfies
\begin{itemize}
\item[$(V0)$] $V\in\cc^{0,1}(\R^3)\cap L^\infty(\R^3)$  and $|V|_\infty<a$.
\end{itemize}
Here we use the notation $|\cdot|_p$ for the various $L^p$-norms. We also require one of the following hypotheses:
\begin{itemize}
\item[$(V1)$] $V$ is $\cc^1$ in a neighborhood of $0$, and $0$ is an isolated local maximum or minimum of $V$.
\item[$(V2)$] $V$ is $\cc^2$ in a neighborhood of $0$, $0$ is an isolated critical point, and there exists a vector space $X\subset\R^3$ such that:
\begin{itemize}
  \item[$(a)$] $V|_X$ has a strict local maximum at $0$;
  \item[$(b)$] $V|_{X^\bot}$ has a strict local minimum at $0$.
\end{itemize}
\end{itemize}
In the case of $(V2)$ we may assume that $\{0\}\ne X\ne\R^3$ so that $0$ is a possibly degenerate saddle point of $V$.

The domain of the quadratic form associated to the Dirac operator is  $H^{\frac12}(\R^3,\C^4)$. This space embeds into the corresponding $L^q$-spaces for $2 \le q \le 3$, and the embedding is locally compact precisely if $q<3$. Therefore the nonlinearity $f(|u|)u$ has subcritical growth if $f(s)s\sim s^{p-1}$ with $2<p<3$, and it has critical growth if $p=3$. In
\eqref{eq:def-ka} below we define for $\la>0$, $p\in(2,3)$ a constant $\bar\ka(V,\la,p)>0$ that appears in the following assumptions when the nonlinearity is critical. Here $F(s):=\int_0^sf(t)t\,dt$ is the primitive of $f(s)s$.
\begin{itemize}
  \item[$(f1)$] $f\in\cc^0[0,\infty)\cap\cc^1(0,\infty)$ satisfies $f(0)=0$ and $f'(s)>0$ for $s>0$.
  \item[$(f2)$] There exist $\la>0$, $p\in(2,3)$, $\ka\in[0,\bar\ka)$ with $\bar\ka=\bar\ka(V,\la,p)$ defined in \eqref{eq:def-ka} such that $f(s)\geq \ka s + \la s^{p-2}$ for $s>0$, and $f'(s)\to\ka$ as $s\to\infty$.
  \item[$(f3)$] There exists $\theta>2$ such that $0<\theta F(s)\le f(s)s^2 + \frac{\theta-2}{3}\ka s^3$ for $s>0$.
\end{itemize}
These conditions imply that $s\mapsto f(s)s$ is strictly increasing and superlinear. Condition $(f3)$ is a weakened Ambrosetti-Rabinowitz condition. If $\ka>0$ then the nonlinearity has critical growth.

\begin{Thm}\label{main result}
Assume that $V$ satisfies $(V0)$ and one of $(V1)$ or $(V2)$. Suppose that $f$ satisfies $(f1)$, $(f2)$ and $(f3)$. Then \eqref{nd2} has a solution $u_\eps$ for $\eps>0$ small. These solutions have the following properties.
\begin{itemize}
  \item[$(i)$] $|u_\eps|$ possesses a global maximum point $x_\eps\in\R^3$ such that $x_\eps \to0$ as $\eps\to0$,
  and
  \[
    |u_\eps(x)|\leq C \exp\Big( -\frac{c}\eps|x-x_\eps| \Big)
  \]
  with $C,c>0$ independent of $\eps$.
  \item[$(ii)$] The rescaled function $U_\eps(x)=u_\eps(\eps x+x_\eps)$ converges as $\eps\to0$ uniformly to a least energy solution $U:\R^3\to\C^4$ of
  \[
    -i \al\cdot\nabla U+a\be U+V(0)U = f(|U|)U.
  \]
\end{itemize}
\end{Thm}

\begin{Rem}
Thus in the subcritical case $\ka=0$ equation \eqref{nd2} always has solutions with shape as in $(i)$ and $(ii)$. We do allow critical growth but the factor $\ka$ cannot be too large. The constant $\bar\ka$ depends on $|V|_\infty$, $\sup V$, $\la$ and $p$. It is bounded away from $0$ by a positive number provided $V$ is bounded away from $-a$ and $\sup V\leq 0$. Moreover $\bar\ka\to0$ as $|V|_\infty\to a$. It is an interesting open problem whether the restriction on $\ka$ can be removed.
\end{Rem}

\section{Preliminaries}\label{sec:Preliminaries}

We write $L^q=L^q(\R^3,\C^4)$ for $q\geq1$ and $H^s=H^s(\R^3,\C^4)$ for $s>0$. Let $D_a=-i\al\cdot\nabla+a\be$ denote the self-adjoint operator on $L^2$ with domain $\cd(D_a)=H^1$. It is well known that the spectrum of $D_a$ is purely continuous and $\si(D_a)=\si_c(D_a)=\R\setminus(-a,a)$. Therefore $L^2$ possesses the orthogonal decomposition
\begin{equation}\label{l2dec}
  L^2=L^+\oplus L^-,\ \ \ \ u=u^++u^-,
\end{equation}
so that $D_a$ is positive definite (resp.\ negative definite) in $L^+$ (resp.\ $L^-$). Now let $E:=\cd(|D_a|^{1/2})$ be the form domain of $D_a$ endowed with the inner product
\[
  \inp{u}{v}=\real\big( |D_a|^{1/2}u,|D_a|^{1/2}v \big)_2
\]
and induced norm $\|\cdot\|$; here $(\cdot,\cdot)_2$ denotes the $L^2$-inner product. This norm is equivalent to the usual $H^{1/2}$-norm, hence $E$ embeds continuously into $L^q$ for all $q\in[2,3]$ and compactly into $L_{loc}^q$ for all $q\in[2,3)$. Clearly $E$ possesses the decomposition
\begin{\equ}\label{Edec}
  E=E^+\oplus E^- \quad \text{with } \ E^{\pm}=E\cap L^{\pm},
\end{\equ}
orthogonal with respect to both $(\cdot,\cdot)_2$ and $\inp{\cdot}{\cdot}$. Since $\si(D_a)=\R\setminus(-a,a)$, one has
\begin{\equ}\label{l2ineq}
  a|u|_2^2\leq\|u\|^2 \quad  \mathrm{for\ all\ }u\in E.
\end{\equ}
The decomposition of $E$ induces also a natural decomposition of $L^q$ for every $q\in(1,\infty)$ as proved in \cite{Ding-Xu-ARMA}.

\begin{Prop}\label{lpdec}
Setting $E^\pm_q:=E^\pm\cap L^q$ for $q\in(1,\infty)$ there holds
\[
  L^q={\rm cl}_q\, E^+_q\op {\rm cl}_q\, E^-_q
\]
with ${\rm cl}_q$ denoting the closure in $L^q$. More precisely, for every $q\in(1,\infty)$ there exists $d_q>0$ such that
\[
  d_q|u^\pm|_q\leq |u|_q \quad {\rm for\ all\ } u\in E\cap L^q.
\]
\end{Prop}

Moreover, the decomposition is invariant when taking derivatives.

\begin{Prop}\label{invariant under derivatives}
For $u\in H^1$ we have $\pa_k u^{\pm}=(\pa_k u)^{\pm}$.
\end{Prop}

\begin{proof}
The Fourier transformation of $D_a$ is given by
\[
(D_au){\hat\ }(\xi)=\left(
\begin{array}{cc}
0& \sum_{k=1}^3\xi_k\si_k \\
\sum_{k=1}^3\xi_k\si_k &0
\end{array} \right)
\hat u
 + \left(
\begin{array}{cc}
a&0\\
0&-a
\end{array} \right)
\hat u,
\]
where $\hat u$, a $\C^4$-valued function, denotes the Fourier transform of $u\in L^2$.
It has been proved in \cite{Ding-Xu-ARMA} that the Fourier transforms of the orthogonal projections $P^\pm:L^2\to L^\pm$ are given by
\[
(P^+u){\hat\ }(\xi)=\Big(\frac12+\frac{a}{2\sqrt{a^2+|\xi|^2}}\Big)
\left(
\begin{array}{cc}
I
&  \Sigma(\xi)
\\[2pt]
 \Sigma(\xi)
&  A(\xi)
\end{array} \right)
\hat u
\]
and
\[
(P^-u){\hat\ }(\xi)=\Big(\frac12+\frac{a}{2\sqrt{a^2+|\xi|^2}}\Big)
\left(
\begin{array}{cc}
A(\xi)& -\Sigma(\xi)
 \\[2pt]
-\Sigma(\xi) & I
\end{array} \right)
\hat u
\]
with $I$ being the $2\times2$ identity matrix and
\[
A(\xi)=\frac{\sqrt{a^2+|\xi|^2}-a}{a+\sqrt{a^2+|\xi|^2}}\cdot I, \quad
\Sigma(\xi)=\sum_{k=1}^3\frac{\xi_k\si_k}{a+\sqrt{a^2+|\xi|^2}}.
\]
The proposition follows from the fact that these matrix operations commute with the multiplication by $i\xi_k$ for $k=1,2,3$.
\end{proof}

The proof of our main results will be achieved by variational methods applied to functionals $J:E\to\R$ of the form
\begin{equation}\label{eq:J}
  J(u) = \frac12\big(\|u^+\|^2-\|u^-\|^2 \big) + \frac12 \int_{\R^3} W(x)|u|^2\,dx - \int_{\R^3}G(x,|u|)\,dx.
\end{equation}
The following reduction process will be very useful.

\begin{Thm}\label{thm:red-couple}
  Let $W\in L^\infty$ satisfy $|W|_\infty<a$ and suppose $G:\R^3\times\R_0^+\to \R$ has the form $G(x,s)=\int_0^sg(x,t)tdt$
  where $g$ is  measurable in $x\in\R^3$, of class $\cc^1$ in $s\in\R^+_0$ and satisfies
  \begin{itemize}
    \item[(i)] $0\le g(x,s)s$ for all $x\in\R^3$;
    \item[(ii)] $g(x,s)s=o(s)$ as $s\to0$ uniformly in $x\in\R^3$;
    \item[(iii)] $0 \le \pa_s\big(g(x,s)s\big) \le Cs$ for all $x\in\R^3$, $s>0$, some $C>0$.
  \end{itemize}
  Then the following hold for $J$ as in \eqref{eq:J}.

  \begin{itemize}
    \item[a)] There exists a $\cc^1$-map $h_J:E^+\to E^-$ such that for $v\in E^+$ and $w\in E^-$
    \[
      DJ(v+w)[\phi]=0 \quad\text{for all $\phi\in E^-$}\qquad\Longleftrightarrow\qquad w=h_J(v)
    \]
    and
    \[
      \|h_J(v)\|^2 \le \frac{2|W|_\infty}{a-|W|_\infty}\|v\|^2+\frac{2a}{a-|W|_\infty}\int_{\R^3} G(x,|v|) dx.
    \]
    \item[b)] Setting $J^{red}:E^+\to\R$, $J^{red}(v):=J(v+h_J(v))$, the sets
    \[
      \cm^+(J):=\{v\in E^+\setminus\{0\}: DJ^{red}(v)[v]=0\}
    \]
    and
    \[
      \cm(J):=\{v+h_J(v)\in E\setminus\{0\}: v\in\cm^+(J)\} = \{u\in E\setminus\{0\}: DJ(u)|_{\R u\oplus E^-}=0\}
    \]
    are $\cc^1$-submanifolds of $E$, diffeomorphic to an open subset of the unit sphere $SE^+=\{v\in E^+:\|v\|=1\}$.
    \item[c)] If $(v_n)_n$ is a Palais-Smale sequence for $J^{red}$ then $\{v_n+h_J(v_n)\}_n$ is a Palais-Smale sequence for $J$.
    \item[d)] If $|g(x,s)| = O\left(|s|^{p-2}\right)$ as $|s|\to\infty$ for some $p\in(2,3)$ then $h_J$ is weakly sequentially continuous.
  \end{itemize}
\end{Thm}

%\todo[inline]{in the whole paper, we have five reduction couples: $(J^{red},h_J)$ (in Theorem 3.3), $(J_\vecmu^{red},h_\vecmu)$
%(for model problems), $(I_\nu^{red},h_\nu)$ (for the limit equation (4.1)), $(\Phi_\eps^{red},h_\eps)$ (for the truncated functional)
%and $(\Phi_y^{red},h_y)$ (before Proposition 5.4)}

The proof of Theorem~\ref{thm:red-couple} is standard. We refer the reader to \cite{Ackermann, Ding-Xu-ARMA, Szulkin-Weth:2010} for this kind of results. The diffeomorphisms to an open subset of $SE^+$ are simply given by $u\mapsto \frac{u^+}{\|u^+\|}$. In the case $W\equiv\nu\in(-a,a)$ the manifold $\cm(J)$ is the Nehari-Pankov manifold associated to $J$. It will be useful that the decomposition $E=E^+\oplus E^-$ is independent of $W$ and does not necessarily correspond to the decomposition of $E$ into the positive and negative eigenspaces associated to $D^2J(0)=P^+-P^-+W(x)$. We call $J^{red}$ the reduced functional, $h_J$ the reduction map, and $(J^{red},h_J)$ the {\it reduction couple} of $J$.

\begin{Rem}\label{rem:red-couple}
  In the setting of Theorem~\ref{thm:red-couple}, for each $v\in SE^+$ the map $\va_v(t)=J^{red}(tv)$ is $\cc^2$ and has at most one critical point $t_v>0$, which is a nondegenerate maximum. There holds:
\[
  \cm^+(J) = \{t_vv: v\in SE^+,\ \va_v'(t_v)=0\}.% \quad\text{and}\quad \cm(J) = \{u\in E\setminus\{0\}: J'(u)|_{\R^+u+E^-}=0\}.
\]
If $G$ grows super-quadratically in $t$ as $t\to\infty$ then $J(tu)\to-\infty$ as $t\to\infty$ and $\va_v(t)$ has a unique maximum for each $v\in SE^+$. Then $\cm(J)$ and $\cm^+(J)$ are diffeomorphic to $SE^+$. It is clear that $\cm(J)$ contains all nontrivial critical points of $J$, and that for $u\in E\setminus\{0\}$ there holds:
\[
   DJ(u)=0 \qquad\Longleftrightarrow\qquad u^-=h_J(u^+) \text{ and } DJ^{red}(u^+)=0
\]
Finally, the infimum of $J$ on $\cm(J)$ can be described as follows:
\begin{equation}\label{eq:def-ga}
\begin{aligned}
  \ga(J) &:= \inf_{u\in\cm(J)}J(u) = \inf_{v\in E^+\setminus\{0\}}\sup_{u\in \R v\op E^-}J(u)\\
         &= \inf_{v\in E^+\setminus\{0\}}\max_{t>0}J^{red}(tv) = \inf_{v\in\cm^+(J)} J^{red}(v)
\end{aligned}
\end{equation}
If $\ga(J)$ is achieved then it is the ground state energy.
\end{Rem}

Theorem~\ref{thm:red-couple} applies in particular to the following functional which depends on the parameters $\vecmu=(\ka,\la,\nu,p)$ with $\ka,\la\ge0$, $|\nu|<a$ and $p\in(2,3)$:
\begin{equation}\label{eq:def-vecmu}
  J_\vecmu(u) = \frac12\big( \|u^+\|^2-\|u^-\|^2 \big) + \frac\nu2|u|_2^2 - \frac\la p|u|_p^p - \frac\ka3|u|_3^3.
\end{equation}
In order to define the constant $\bar\ka$ from Theorem~\ref{main result} let
\begin{equation}\label{eq:defS}
  S:=\inf_{0\ne u\in H^1}\frac{|\nabla u|_2^2}{|u|_6^2}
\end{equation}
be the best constant for the embedding $H^1(\R^3,\C^4)\hookrightarrow L^6(\R^3,\C^4)$. Then we define for $V$ as in $(V0)$, $\la>0$, $p\in(2,3)$ as in $(f2)$, $(f3)$:
%We shall prove that $\ga(\nu,\la,p)>0$ is always achieved.
\begin{equation}\label{eq:def-ka}
  \bar\ka := \left(\frac{a^2-|V|_\infty^2}{a^2}\right)^{\frac34}S^{\frac34}\big(6\ga(J_{\vecmu_V})\big)^{-\frac12}\qquad \text{with }\vecmu_V:=(0,\la,\sup V,p).
\end{equation}%\todo[inline]{change $\vecmu_0$ to $\vecmu_V$, since $\vecmu_0$ is re-defined in the proof of Proposition 4.1}

The following technical result will be needed later.

\begin{Lem}\label{key lemma}
  For $v\in E^+\setminus\{0\}$ the function $H(t)=I(tv)-\frac{t}2I'(tv)[v]$ is of class $\cc^1$ and satisfies $H'(t)>0$ for all $t>0$.
\end{Lem}

\begin{proof}
We set $\va_v(t)=I(tv)$ so that $H(t)=\va_v(t)-\frac{t}2\va_v'(t)$. Since
\[
  H'(t) = \frac12\va_v'(t)-\frac{t}2\va_v''(t) = \frac1{2t}\big[ \va_{tv}'(1)-\va_{tv}''(1) \big],
\]
it is sufficient to check that $\va_{v}'(1)-\va_{v}''(1)>0$ for all $v\in E^+\setminus\{0\}$. Setting $K(u)=\int_{\R^3}G(x,|u|)\,dx$, we have by the definition of $h_J$
\begin{\equ}\label{X1}
  -\inp{h_J(v)}{\phi}+\real\int_{\R^3}W(x)(v+h_J(v))\cdot\ov{\phi} \,dx- K'(v+h_J(v))[\phi]=0
\end{\equ}
for all $\phi\in E^-$. It follows for $z_v=v+h_J(v)$ and $w_v=h_J'(v)[v]-h_J(v)$ that
\begin{\equ}\label{X2}
\va_v'(1) = \|v\|^2+\real\int_{\R^3}W(x)z_v\cdot\ov{v}\,dx-K'(z_v)[v] = J'(z_v)[z_v+w_v].
\end{\equ}
Since \eqref{X1} is valid for all $v\in E^+$, differentiating yields for all $\phi\in E^-$:
\[
  0 = -\inp{h_J'(v)[v]}{\phi}+\real\int_{\R^3}W(x)(v+h_J'(v)[v])\cdot\ov{\phi}\,dx - K''(v+h_J(v))[v+h_J'(v)[v],\phi]\,.
\]
Choosing $\phi=h_J'(v)[v]$ in the above identity, so that $z_v+w_v=v+\phi$, we get
\[
  \begin{aligned}
  \va_v''(1) &= \|v\|^2+\real\int_{\R^3}W(x)(v+h_J'(v)[v])\cdot\ov{v}\,dx - K''(z_v)[z_v+w_v,v]\\
    &= \|v\|^2-\|\phi\|^2+\int_{\R^3}W(x)|v+\phi|^2\, dx - K''(z_v)[z_v+w_v,v+\phi]\\
    &= J''(z_v)[z_v+w_v,z_v+w_v]\\
    &= \|v\|^2 - \|h_J(v)\|^2+\int_{\R^3}W(x)|z_v|^2\,dx - K''(z_v)[z_v,z_v]\\
    &\hspace{1cm} +2\left( -\inp{h_J(v)}{w_v} + \real\int_{\R^3}W(x)z_v\cdot\ov{w_v}\,dx - K''(z_v)[z_v,w_v] \right)\\
    &\hspace{1cm} +\left( -\|w_v\|^2 + \int_{\R^3}W(x)|w_v|^2\,dx - K''(z_v)[w_v,w_v] \right)\\
    &= \va_v'(1) + \big( K'(z_v)[z_v] - K''(z_v)[z_v,z_v] \big) + 2\big( K'(z_v)[w_v] - K''(z_v)[z_v,w_v] \big)\\
    &\hspace{1cm} - K''(z_v)[w_v,w_v] - \|w_v\|^2 + \int_{\R^3}W(x)|w_v|^2\,dx.
  \end{aligned}
\]
Finally we obtain:
\[
  \va_v'(1)-\va_v''(1) \ge \int_{\R^3}G'(x,|z_v|)|w_v|^2 + G''(x,|z_v|)|z_v| \Big( |z_v|+ \frac{\real\, z_v\cdot \ov{w_v}}{|z_v|} \Big)^2\,dx > 0
\]
\end{proof}
\section{The limit problem}\label{sec:limit}
For $|\nu|<a$ the problem
\begin{\equ}\label{eq:limit problem}
  -i\al\cdot\nabla u + a\be u + \nu u = f(|u|)u, \quad u\in E,
\end{\equ}
appears as limit equation of \eqref{nd2}. We begin with the model case
%\begin{\equ}\label{eq:limit model}
 \[
   -i\al\cdot\nabla u+a\be u+\nu u= \la|u|^{p-2}u+\ka|u|u \quad u\in E.
\]
%\end{\equ}
and recall the associated energy functional $J_\vecmu$ from \eqref{eq:def-vecmu} with $\vecmu=(\ka,\la,\nu,p)$ and $\ka,\la,p$ from $(f2)$, $(f3)$.

\begin{Prop}\label{prop:inf_Jmu}
The infimum $\ga(J_\vecmu)$ is attained provided $\nu$ satisfies
\begin{\equ}\label{eq:inf_Jmu}
  \Big(\frac{a^2}{a^2-\nu_-^2}\Big)^{\frac32}\cdot\ka^2\cdot\ga(J_\vecmu) < \frac{S^{\frac32}}{6},
\end{\equ}
where $\nu_-=\min\{0,\nu\}$. %\todo{add the definition of $\nu_-$}
\end{Prop}

\begin{proof}
We only give the proof for $\ka>0$ since the subcritical case $\ka=0$ is much easier. Let $(J_{\vecmu}^{red}, h_{\vecmu})$ denote the reduction couple of $J_{\vecmu}$ and let $(v_n)_n$  be a minimizing sequence for $J_{\vecmu}^{red}$ in $\cm^+(J_{\vecmu})$. Setting $u_n=v_n+h_{\vecmu}(v_n)$ it is not difficult to check that $(u_n)_n$ is bounded in $E$, hence it is either vanishing or non-vanishing up to a subsequence (see \cite{Lions}).

If $(u_n)_n$ has a non-vanishing subsequence then we are done, so let us assume to the contrary that $(u_n)_n$ is vanishing, hence $|u_n|_p\to0$. We first show that this implies
\begin{\equ}\label{eq:Jmu0-1}
\ga(J_{\vecmu}) \ge \ga(J_{\vecmu_0})\qquad\text{where }\ \vecmu_0=(\ka,0,\nu,p).
\end{\equ}
In order to see this let $t_n>0$ be defined by $t_nv_n\in\cm^+(J_{\vecmu_0})$. Observe that  $\|v_n\|$ is bounded away from $0$ and the nonlinearity in $J_{\vecmu_0}$ is super-quadratic, so that $(t_n)_n$ is bounded. Theorem~\ref{thm:red-couple}~d) now implies $|h_{\vecmu_0}(t_nv_n)|_p\to0$ where $h_{\vecmu_0}$ is the reduction map for $J_{\vecmu_0}$. Now \eqref{eq:Jmu0-1} follows from
\[
\begin{aligned}
  \ga( J_{\vecmu_0}) &\le J_{\vecmu_0}(t_nv_n+h_{\vecmu_0}(t_nv_n))\\
       &= J_\vecmu(t_nv_n+h_{\vecmu_0}(t_nv_n)) + o_n(1) \le J_{\vecmu}^{red}(v_n) +o_n(1) = \ga( J_{\vecmu})+o_n(1).
\end{aligned}
\]

Next we show that
\begin{\equ}\label{eq:Jmu0-2}
  J_{\vecmu_0}^{red}(v) \ge \frac1{6\ka^2}\left(\frac{\|v\|^2+\nu|v|_2^2}{|v|_3^2}\right)^3
     \qquad\text{for all } v\in\cm^+(J_{\vecmu_0})\,.
\end{\equ}
For this we consider the functional
\[
  I: E\setminus\{0\}\to\R, \quad u\mapsto \frac{\|u^+\|^2-\|u^-\|^2+\nu|u|_2^2}{|u|_3^2}.
\]
For any $v\in E^+$ it is easy to see by a direct argument that $\sup_{w\in E^-}I(v+w)>0$ is achieved by some $w_v\in E^-$. Moreover, for any $c>0$ the set $\{w\in E^-: I(v+w)\ge c\}$ is strictly convex because
\[
  w\mapsto \|v\|^2 - \|w\|^2 + \nu|v+w|_2^2 - c|v+w|_3^2
\]
is strictly concave on $E^-$. This also uses $|\nu|<a$. Hence $w_v$ is the unique critical point of $w\mapsto I(v+w)$. On the other hand, for $v\in \cm^+(J_{\vecmu_0})$, we have
\begin{\equ}\label{eq:Jmu0-3}
  0 = DJ_{\vecmu_0}^{red}(v)[v]
     = \|v\|^2 - \|h_{\vecmu_0}(v)\|^2 + \nu|v+h_{\vecmu_0}(v)|_2^2 - \ka|v+h_{\vecmu_0}(v)|_3^3,
\end{\equ}
hence
\[
  J_{\vecmu_0}^{red}(v) = J_{\vecmu_0}^{red}(v) - \frac12 DJ_{\vecmu_0}^{red}(v)[v]
    = \frac{\ka}6 |v + h_{\vecmu_0}(v)|_3^3.
\]
A direct calculation gives
\[
  DI\big(v+h_{\vecmu_0}(v)\big)\big|_{E^-} = 0 \quad\text{and}\quad I\big(v+h_{\vecmu_0}(v)\big) > 0
\]
which implies $h_{\vecmu_0}(v) = w_v$. Now \eqref{eq:Jmu0-2} follows, using \eqref{eq:Jmu0-3} once more:
\[
  J_{\vecmu_0}^{red}(v) = \frac{\ka}6 |v + h_{\vecmu_0}(v)|_3^3 = \frac1{6\ka^2}I^3(v+h_{\vecmu_0}(v))
          \ge \frac1{6\ka^2}I^3(v)
\]

Finally, the proposition follows from \eqref{eq:Jmu0-1},  \eqref{eq:Jmu0-2} and
\begin{equation}\label{eq:Jmu0-4}
  \frac{\|v\|^2+\nu|v|_2^2}{|v|_3^2} \ge  \Big(\frac{a^2-\nu_-^2}{a^2}\Big)^{\frac12}S^{\frac12}
     \qquad\text{for all } v\in\cm^+(J_{\vecmu_0})\,.
\end{equation}
For the proof of \eqref{eq:Jmu0-4} we pass to the Fourier domain and recall from \cite{Ding-Xu-ARMA} that
\[
  \|u\|^2 = \int_{\R^3}(a^2+|\xi|^2)^{\frac12}|\hat u|^2\,d\xi\qquad\text{for all }u\in E.
\]
Since $|\nu|<a$ we have
%\todo{in case $\nu<0$, the infimum of the left hand side divided by $|t|$ can be obtained just by taking the derivative and calculate the unique extremum}
\[
  (a^2+t^2)^{\frac12}+\nu \ge \left(\frac{a^2-\nu_-^2}{a^2}\right)^{\frac12}|t|\quad\text{for all } t\in\R
\]
which implies for $v\in E^+\setminus\{0\}$:
\[
\begin{aligned}
   \frac{\|v\|^2+\nu|v|_2^2}{|v|_3^2} &= \frac{\int_{\R^3}[(a^2+|\xi|^2)^{\frac12}+\nu]\cdot|\hat v|^2\,d\xi}{|v|_3^2}
     \ge \left(\frac{a^2-\nu_-^2}{a^2}\right)^{\frac12}\frac{\int_{\R^3}|\xi||\hat u|^2\,d\xi}{|u|_3^2}\\
    & \ge \left(\frac{a^2-\nu_-^2}{a^2}\right)^{\frac12} S^{\frac12}
\end{aligned}
\]
Here the last inequality follows from
\[
  \frac{\int_{\R^3}|\xi|^2|\hat u|^2 d\xi}{|u|_6^2}=\frac{|\widehat{\nabla u}|_2^2}{|u|_6^2}
   =\frac{|\nabla u|_2^2}{|u|_6^2}\ge S \quad \text{for all } u\in H^1(\R^3,\C^4)
\]
and the Calder\'on-Lions interpolation theorem (see \cite{RS}).
%\todo{reference for interpolation theorem?}
\end{proof}

Now we consider the energy functional $I_\nu: E\to\R$ associated to \eqref{eq:limit problem} given by
\begin{\equ}\label{limit problem functional}
  I_\nu(u) = \frac12\big( \|u^+\|^2-\|u^-\|^2 \big) +\frac\nu2|u|_2^2-\int_{\R^3} F(|u|)dx.
\end{\equ}
The hypotheses $(f1)-(f3)$ imply that $I_\nu$ satisfies the assumptions of Theorem~\ref{thm:red-couple} for $|\nu|<a$.

\begin{Lem}\label{critical attained}
  If $\nu_0\in(-a,a)$ satisfies \eqref{eq:inf_Jmu} then $\ga(I_\nu)$ is achieved for all $\nu\in(-a,\nu_0]$. Moreover, the map $\nu\mapsto \ga(I_\nu)$ is continuous and strictly increasing.
\end{Lem}

%As a consequence $\ga(I_\nu) = \inf\big\{J_\nu(u):\, u\in E\setminus\{0\},\, DJ_\nu(u)=0\big\}$.

\begin{proof}
For $\nu\in(-a,\nu_0]\subset(-a,a)$ assumption $(f3)$ implies $I_{\nu}\leq J_{\vecmu_1}\le J_{\vecmu_2}$, where $\vecmu_1=(\ka,\la,\nu_0,p)$ and $\vecmu_2=(0,\la,\nu_0,p)$. It follows that $\ga(I_{\nu}) \leq \ga(J_{\vecmu_1})\leq\ga(J_{\vecmu_2})$.
%\begin{equation}\label{def:kappabar}
%  \ga(J_\nu) \leq \ga(J_\vecmu)\leq\ga(J_{\vecmu_1}),
%\end{equation}
A similar argument as in the proof of Proposition~\ref{prop:inf_Jmu} implies the existence of a nontrivial critical point $u_{\nu}$ for $I_{\nu}$ such that $u_{\nu}^+$ is the minimizer for $I_{\nu}^{red}$ on $\cm^+(I_{\nu})$.

In order to prove the monotonicity of $\ga(\nu)$ we consider $-a<\nu_1<\nu_2\leq\nu_0$. Let $u\in\cm(I_{\nu_2})$ be a minimizer for $\ga(I_{\nu_2})$ and define $s>0$ by $su^+\in\cm^+(I _{\nu_1})$. Then we have, with $(I_{\nu_1}^{red},h_{\nu_1})$ denoting the reduction couple for $I_{\nu_1}$ and $u_1:=t_1 u^++h_{\nu_1}(su^+)\in\cm(I_{\nu_1})$:
\[
\aligned
\ga(I_{\nu_1})
   &\le I_{\nu_1}^{red}(su^+) = I_{\nu_1}(u_1) = I_{\nu_2}(u_1) - \frac{\nu_2-\nu_1}2|u_1|_2^2
     \le  I_{\nu_2}^{red}(t_1u^+) - \frac{\nu_2-\nu_1}2|u_1|_2^2 \\
   &\le \max_{t>0}I_{\nu_2}^{red}(tu^+) - \frac{\nu_2-\nu_1}2|u_1|_2^2
     =\ga(I_{\nu_2}) - \frac{\nu_2-\nu_1}2|u_1|_2^2.
\endaligned
\]
Choosing a minimizer $v\in\cm(I_{\nu_1})$ for $\ga(I_{\nu_1})$, defining $t>0$ by $tv^+\in\cm^+(I _{\nu_2})$, and setting $u_2:=tv^++h_{\nu_1}(tv^+)\in\cm(I_{\nu_2})$, an analogous argument shows that
\[
  \ga(I_{\nu_2}) \le \ga(I_{\nu_1}) + \frac{\nu_2-\nu_1}2|u_2|_2^2.
\]
For the continuity of $\ga(\nu)$ it remains to prove that $s,t$ are bounded for $\nu_1,\nu_2$ in a compact subset of $(-a,\nu_0]$ because then $|\ga(I_{\nu_2})-\ga(I_{\nu_1})|=O(\nu_2-\nu_1)$. This follows for $s$ from
\[
  0<I_{\nu_1}^{red}(su^+)
  %&=\frac12\big( \|tw\|^2-\|\msj_{\nu_j}(tw)\|^2\big)
%+\frac{\nu_j}2\big|tw+\msj_{\nu_j}(tw)\big|_2^2-\int_{\R^3}F\big(|tw+\msj_{\nu_j}(tw)|\big)dx\\
\leq\frac{s^2}2\big( \|u^+\|^2+\nu_1|u^+|_2^2 \big) - \frac{d_p\la}ps^p|u^+|_p^p.
\]
where $d_p>0$ is from  Proposition \ref{lpdec}. The bound for $t$ is proved analogously.
\end{proof}

%%%
%%%

\section{The truncated problem}\label{sec:truncation}

For a subset $\Lam\subset\R^3$, let $\Lam^c$ denote its complement, and $\Lam^\eps:=\big\{x\in\R^3:\, \eps x\in\Lam \big\}$, $\eps>0$. By the change of variables $x\mapsto \eps x$ and setting $V_\eps(x)=V(\eps x)$, the singularly perturbed problem \eqref{nd2} is equivalent to
\begin{\equ}\label{Dirac0}
  -i  \al\cdot\nabla u + a\be u  +V_\eps(x)u = f(|u|)u.
\end{\equ}
In the sequel, we will modify the function $f$ similar to \cite{Del-Pino, Del-Pino2}. For
\begin{\equ}\label{eq:def-de0}
  0 < \de_0 \le \frac{a-|V|_\infty}{4},
\end{\equ}
we define $\tilde{f}=\tilde{f}_{\de_0}\in C^1(\R^+_0)$ by $\tilde{f}(0)=0$ and
\[
  \frac{d}{ds}\big(\tilde f (s)s \big) = \min\big\{f'(s)s+f(s),\,\de_0\big\}.
\]
In the subcritical case $\ka=0$ of Theorem~\ref{main result} the choice $\de_0 = \frac{a-|V|_\infty}{4}$ will be fine. For the critical case $\ka>0$ we need to make $\de_0$ smaller in the course of the proof. Let $\tilde F(s)=\int_0^s\tilde f(t)t\,dt$ be the primitive of $\tilde{f}(s)s$. By our assumptions on $V$ there exists $R_1>0$ so that
\begin{\equ}\label{pa V}
  \nabla V(x)\notin \R x\quad \text{for all $x\in \R^3$ with $|x|=R_1$ and $V(x)=V(0)$,}
\end{\equ}
see \cite{DPR}. We define the cut-off function $\chi:\R^3\to[0,1]$ by
\begin{\equ}\label{chi}
\chi(x)=\begin{cases}
          1, & \mbox{if } |x|< R_1 \\
          \frac{2R_1-|x|}{R_1}, & \mbox{if } R_1\le |x|< 2R_1 \\
          0, & \mbox{if } |x|\ge 2R_1.
        \end{cases}
\end{\equ}
and consider
\[
  g(x,s)=\chi(x) f(s) + \big(1-\chi(x)\big)\tilde f(s)
\]
as well as
\[
  G(x,s) = \int_0^sg(x,t)tdt = \chi(x)F(s)+\big(1-\chi(x)\big)\tilde F(s).
\]
For later use, associated to the above notations,
we denote $B_1=B(0,R_1)$ and $B_2=B(0,2R_1)$
the open balls in $\R^3$ of radius $R_1$ and $2R_1$.
The following lemma is easy to prove.

\begin{Lem}\label{mod}
  The function $G(x,s)$ satisfies the conditions $(i)-(iii)$ from Theorem~\ref{thm:red-couple}.
\end{Lem}

We will consider the truncated problem
\begin{\equ}\label{truncated dirac}
  -i \al\cdot\nabla u+a\be u+V_\eps(x)u = g_\eps(x,|u|)u, \quad u\in E
\end{\equ}
where we write $g_\eps(x,s)=g(\eps x,s)$; we also use the notations $\chi_\eps$ and $G_\eps$ for the dilations of $\chi$ and $G$, respectively. The corresponding energy functional is
\[
  \Phi_\eps(u) = \frac12\big(\|u^+\|^2-\|u^-\|^2 \big) + \frac12 \int_{\R^3} V_\eps(x)|u|^2\,dx - \int_{\R^3} G_\eps(x,|u|)\,dx.
\]
As a direct consequence of Lemma \ref{mod}, we can introduce
$(\Phi_\eps^{red},h_\eps)$ as the reduction couple of $\Phi_\eps$.

In order to establish a compactness result for $\Phi_\eps$, we first prove a bound for Palais-Smale sequences of $\Phi_\eps$ that is uniform in $\eps$.

\begin{Lem}\label{boundedness}
  For $c\in\R$ fixed, $(PS)_c$-sequences of $\Phi_\eps$ are bounded uniformly in $\eps$.
\end{Lem}

\begin{proof}
Given a $(PS)_c$-sequence $(u_n)_n$ for $\Phi_\eps$ we have by our conditions on $f$:
\[
\aligned
&\int_{\R^3} \chi_\eps(x)f(|u_n|)|u_n|\cdot|u_n^+-u_n^-| dx \\
&\qquad \leq \Big(\int_{\R^3}\chi_\eps(x)\big(f(|u_n|)|u_n|\big)^{\frac32}dx
\Big)^{\frac23}\cdot \big|u_n^+-u_n^-\big|_3
+\de_0\int_{\R^3}\chi_\eps(x)|u_n|\cdot|u_n^+-u_n^-|dx\\
&\qquad \leq C_\theta\Big(\int_{\R^3}\chi_\eps(x)\big(f(|u_n|)|u_n|^2
-2F(|u_n|)\big)dx\Big)^{\frac23}  \|u_n\|
+\de_0 |u_n|_2^2, %\int_{\R^3}\chi_\eps(x)|u_n|\cdot|u_n^+-u_n^-|dx,
\endaligned
\]
where $C_\theta>0$ only depends on the constant $\theta>2$ in $(f2)$. It follows from \eqref{eq:def-de0} that
\[
\aligned
\Big( 1-\frac{|V|_\infty}a\Big)\|u_n\|^2 &\leq\Phi_\eps'(u_n)[u_n^+-u_n^-]+\int_{\R^3} g_\eps( x,|u_n|)|u_n|\cdot|u_n^+-u_n^-|dx\\
&\leq C_\theta\Big(2\Phi_\eps(u_n)-\Phi_\eps'(u_n)[u_n]\Big)^{\frac23}\|u_n\| +2 \de_0|u_n|_2^2+o(\|u_n\|).
\endaligned
\]
Now the lemma follows using \eqref{l2ineq}:
\begin{\equ}\label{ps bdd}
\Big(1-\frac{|V|_\infty+2\de_0}a\Big)\|u_n\|^2 \le C_\theta\big(2(c+o(1))+o(\|u_n\|)\big)^{\frac23}\|u_n\|+o(\|u_n\|).
\end{\equ}
\end{proof}

Now we can prove the Palais-Smale condition for $\Phi_\eps$. Recall that the nonlinearity $G$ in $\Phi_\eps$ depends on a constant $\de_0$; see \eqref{eq:def-de0}.

\begin{Prop}\label{PS condition}
If
\[
  \ka^2\cdot c_0< \left(\frac{a^2-|V|_\infty^2}{a^2}\right)^{\frac32}\cdot \frac{S^{\frac32}}6,
\]
then there exists $\de_0>0$ such that the truncated functional $\Phi_\eps$ satisfies the $(PS)_c$-condition for all $c\leq c_0$, all $\eps>0$.
\end{Prop}

\begin{proof}
We choose $\de_0>0$ so that
\[
  \left( \frac{a^2-|V|_\infty^2}{a^2} \right)^{\frac32}\frac{S^{\frac32}}6
   > \left( \frac{a^2-(|V|_\infty+\de_0)^2}{a^2} \right)^{\frac32}\frac{S^{\frac32}}6
   > \ka^2\cdot c_0.
 \]
Let $(u_n)_n$ be a $(PS)_c$-sequence for $\Phi_\eps$ with $c\leq c_0$, any $\eps>0$. By Lemma \ref{boundedness} there exists $u\in E$ such that, along a subsequence, $u_n\weakto u$ in $E$ and $u_n\to u$ strongly in $L_{loc}^q$ for $q\in[2,3)$. We want to show that $u_n\to u$ strongly in $E$.

Set $z_n=u_n-u$ so that $z_n\weakto 0$ in $E$ and $\|u_n^\pm\|^2=\|u^\pm\|^2+\|z_n^\pm\|^2+o_n(1)$. Note that
\[
  \lim_{s\to0}\tilde f(s)=\lim_{s\to\infty}\frac{\tilde f(s)}s=0\qquad\text{and}\qquad
  \lim_{s\to0} f(s)=\lim_{s\to\infty}\frac{ f(s)}s-\ka=0.
\]
By the Brezis-Lieb lemma (see for instance \cite[Lemma~1.32]{Willem}) there holds
\[
  \int_{\R^3}G_\eps(x,|u_n|)
    = \int_{\R^3}G_\eps(x,|u|)+\int_{\R^3}\big(1-\chi_\eps(x)\big)\tilde F(|z_n|)+\frac\ka3\int_{\R^3}\chi_\eps(x)|z_n|^3+o_n(1),
\]
and
\[
  \int_{\R^3}g_\eps(x,|u_n|)|u_n|^2
   = \int_{\R^3}g_\eps(x,|u|)|u|^2 +\int_{\R^3}\big(1-\chi_\eps(x)\big)\tilde f(|z_n|)|z_n|^2+\ka\int_{\R^3}\chi_\eps(x)|z_n|^3+o_n(1).
\]
Therefore
\[
  \Phi_\eps(u_n)=\Phi_\eps(u)+\Phi_\eps(z_n)+o_n(1),
\]
and
\[
  D\Phi_\eps(u_n)[u_n]=D\Phi_\eps(u)[u]+D\Phi_\eps(z_n)[z_n]+o_n(1).
\]
Obviously, $D\Phi_\eps(u)=0$, hence $D\Phi_\eps(z_n)[z_n]=o_n(1)$. We claim that
\begin{equation}\label{claim0}
  D\Phi_\eps(z_n)\to 0 \ \text{as }n\to\infty.
\end{equation}
In fact, consider $\va\in E$ with $\|\va\|\le 1$ and set $g^1(x,s)=g(x,s)-\ka \chi(x) s$. We have
\begin{eqnarray}\label{claim1}
D\Phi_\eps(u_n)[\va]
  &=& \inp{u_n^+-u_n^-}{\va} +\real\int_{\R^3} V_\eps(x)u_n\cdot\bar\va - \real\int_{\R^3} g_\eps(x,|u_n|)u_n\cdot\bar\va  \nonumber \\
%  &= \inp{z_n^++u^+}{\va^+}-\inp{z_n^-+u^-}{\va^-} + \real\int_{\R^3} V_\eps(x)(z_n+u)\cdot\bar\va \\
%  &\hspace{1cm} -\real\int_{\R^3}g_\eps(x,|z_n+u|)(z_n+u)\cdot\bar\va \\
  &=& \inp{z_n^+}{\va^+}-\inp{z_n^-}{\va^-} + \inp{u^+}{\va^+}-\inp{u^-}{\va^-} \nonumber \\
&  &\quad +\real\int_{\R^3} V_\eps(x)z_n\cdot\bar\va + \real\int_{\R^3} V_\eps(x)u\cdot\bar\va  \nonumber \\
&  &\quad -\real\int_{\R^3} g_\eps^1(x,|z_n|)z_n\cdot\bar\va - \real\int_{\R^3} g_\eps^1(x,|u|)u\cdot\bar\va \nonumber\\
&  &\quad -\ka\cdot\real\int_{\R^3}\chi_\eps(x)|z_n+u|(z_n+u)\cdot\bar\va + o_n(\|\va\|)
\end{eqnarray}%\todo{rewritten?}
where we used $u_n=z_n+u$ and $D\Phi_\eps(u)=0$. The estimate for the subcritical part
\[
  \real\int_{\R^3} g_\eps^1(x,|u_n|)u_n\cdot\bar\va-\real\int_{\R^3} g_\eps^1(x,|z_n|)z_n\cdot\bar\va - \real\int_{\R^3} g_\eps^1(x,|u|)u\cdot\bar\va
   = o_n(\|\va\|)
\]
follows from a standard argument in \cite[Lemma 7.10]{Ding2}. To estimate the last integral in \eqref{claim1}, we set $\psi_n:=|z_n+u|(z_n+u)-|z_n|z_n-|u|u$ and observe $|\psi_n|\leq 2|z_n|\cdot|u|$. By the Egorov theorem there exists $\Theta_\si\subset B_2^\eps$ such that $\meas(B_2^\eps\setminus\Theta_\si)<\si$ and $z_n\to0$ uniformly on $\Theta_\si$ as $n\to\infty$. Thus, by the H\"older inequality, we have
\[
\aligned
&\int_{\R^3}\chi_\eps(x)|\psi_n|\cdot|\va| \le \int_{\Theta_\si}|\psi_n|\cdot|\va| + \int_{B_2^\eps\setminus\Theta_\si}|\psi_n|\cdot|\va|\\
&\qquad \leq\int_{\Theta_\si}|\psi_n|\cdot|\va| + 2\left( \int_{B_2^\eps\setminus\Theta_\si}|z_n|^3 \right)^{\frac13}
\cdot\left( \int_{B_2^\eps\setminus\Theta_\si}|u|^3 \right)^{\frac13}\cdot\left( \int_{B_2^\eps\setminus\Theta_\si}|\va|^3 \right)^{\frac13}.
\endaligned
\]
The first integral in the last line converges to $0$ as $n\to\infty$ and the remaining integrals go to $0$ uniformly in $n$ as $\si\to0$. This shows
\[
\int_{\R^3}\chi_\eps(x)|\psi_n|\cdot|\va|=o_n(\|\va\|) \quad \text{as } n\to\infty
\]
and consequently, using again $D\Phi_\eps(u)=0$,
\[
\aligned
  D\Phi_\eps(u_n)[\va]
    &= \inp{z_n^+}{\va^+}-\inp{z_n^-}{\va^-} + \inp{u^+}{\va^+} - \inp{u^-}{\va^-}\\
    &\hspace{1cm}
           +\real\int_{\R^3} V_\eps(x)z_n\cdot\bar\va + \real\int_{\R^3} V_\eps(x)u\cdot\bar\va\\
    &\hspace{1cm}
           -\real\int_{\R^3} g_\eps^1(x,|z_n|)z_n\cdot\bar\va - \real\int_{\R^3} g_\eps^1(x,|u|)u\cdot\bar\va \\
    &\hspace{1cm}
           -\ka\cdot\real\int_{\R^3}\chi_\eps(x)|z_n|z_n\cdot\bar\va - \ka\cdot\real\int_{\R^3}\chi_\eps(x)|u|u\cdot\bar\va +  o(\|\va\|)\\
    &= D\Phi_\eps(z_n)[\va] + D\Phi_\eps(u)[\va] + o_n(\|\va\|)\\
    &= D\Phi_\eps(z_n)[\va] + o_n(\|\va\|)
\endaligned
\]
It follows that $D\Phi_\eps(z_n) \to 0$ as $n \to \infty$ as claimed in \eqref{claim0}. Now $D\Phi_\eps(z_n)[z_n^+-z_n^-] = o_n(1)$ reads as
\[
\aligned
&\|z_n\|^2+\real\int_{\R^3}V_\eps(x)z_n\cdot\ov{(z_n^+-z_n^-)}\\
&\quad
 = \int_{\R^3}\big(1-\chi_\eps(x)\big)\tilde f(|z_n|)z_n \cdot\ov{(z_n^+-z_n^-)}
    + \ka\cdot\real\int_{\R^3}\chi_\eps(x)|z_n|z_n \cdot \ov{(z_n^+-z_n^-)} + o_n(1).
\endaligned
\]
Then, by using the fact $\tilde f(s)\leq\de_0$ and \eqref{eq:Jmu0-4}, we obtain
\[
\left( \frac{a^2-(|V|_\infty+\de_0)^2}{a^2} \right)^{\frac12}S^{\frac12}\left( \int_{\R^3}\chi_\eps(x)|z_n|^3dx \right)^{\frac23}
  \le \ka\cdot\int_{\R^3}\chi_\eps(x)|z_n|^3 dx + o_n(1).
\]
If $b:=\lim_{n\to\infty}\int_{\R^3}\chi_\eps(x)|z_n|^3 dx=0$ then $\|z_n\|=o_n(1)$ and $u_n\to u$ strongly in $E$, as claimed. Suppose to the contrary that $b>0$ so that
\[
  \left( \frac{a^2-(|V|_\infty+\de_0)^2}{a^2} \right)^{\frac32}S^{\frac32}
    \le \ka^3 \cdot \int_{\R^3}\chi_\eps(x)|z_n|^3 dx = \ka^3 \cdot b + o_n(1).
\]
In case $\ka=0$, this is a contradiction. In case $\ka>0$, using
\[
  \Phi_\eps(u) = \int_{\R^3} \frac12 g_\eps(x,|u|)|u|^2-G_\eps(x,|u|) \geq 0,
\]
as well as $\Phi_\eps(u_n)\geq\Phi_\eps(z_n)+o_n(1)$ and $D\Phi_\eps(z_n)[z_n] = o_n(1)$, we obtain the contradiction
\[
\ka^2\cdot c + o_n(1)\geq
  \frac{\ka^3}6\cdot b+o_n(1)
  \geq \left( \frac{a^2-(|V|_\infty+\de_0)^2}{a^2} \right)^{\frac32}\frac{S^{\frac32}}6+o_n(1).
\]
\end{proof}

We finish this section with a couple of notations that will be of use later. For simplicity, when $\nu$ belongs to the range of $V(x)$ that is $\nu\in\{V(x):\, x\in\R^3\}$,
we denote $\nu_0=V(0)$ and correspondingly
\begin{\equ}\label{notations}
I_{\nu_0}=I_{V(0)}, \quad I_{\nu_0}^{red}=I_{V(0)}^{red}, \quad  \ga(I_{\nu_0})=\ga(I_{V(0)}).
\end{\equ}
Moreover, given arbitrarily $y\in\R^3$, we can define the functional $\Phi_y: E\to\R$,
\[
\Phi_y(u)=\frac12\big( \|u^+\|^2-\|u^-\|^2 \big)+\frac{V(y)}2|u|_2^2-\int_{\R^3}G(y,|u|)dx,
\]
and $(\Phi_y^{red}, h_y)$ the reduction couple associated to $\Phi_y$.
Plainly, the critical point of $\Phi_y$ are solutions of the problem
\[
-i\al\cdot\nabla u+ a\be u + V(y)u=g(y,|u|)u.
\]
When $y\in B_1$, we have $\Phi_y=I_{V(y)}$ and $h_y=h_{V(y)}$.
Let us point out that, by virtue of \cite[Lemma 4.3]{Ding-Xu-Trans}, we can conclude the
following splitting type result, whose proof is postponed to the appendix.
\begin{Prop}\label{h-vr converges}
For $y\in\R^3$, let us define the functional $\Phi_{\eps,y}: E\to\R$,
\[
\Phi_{\eps,y}(u)=\frac12\big(
\|u^+\|^2-\|u^-\|^2 \big) +\frac12 \int_{\R^3} V(\eps x+y)|u|^2dx
-\int_{\R^3} G(\eps x+y,|u|)dx ,
\]
and $(\Phi_{\eps,y}^{red},h_{\eps,y})$ the associated reduction couple, we have that
\begin{itemize}
\item[$(1)$] let $\{y_\eps\}\subset\R^3$ be such that $y_\eps\to y$ for some $y\in\R^3$
then, up to a subsequence,
$h_{\eps,y_\eps}(w)\to h_y(w)$ as $\eps\to0$ for each $w\in E^+$;

\item[$(2)$] let $\{y_\eps\}\subset\R^3$ be such that $y_\eps\to y$ for some $y\in\R^3$
and let $\{w_\eps\}\subset E^+$ be such that $w_\eps\weakto w$
for some $w\in E^+$ then, up to a subsequence,
\[
\|h_{\eps,y_\eps}(w_\eps)-h_{\eps,y_\eps}(w_\eps-w)-h_y(w)\|=o_\eps(1)
\]
as $\eps\to0$;

\item[$(3)$] let $\{y_\eps\}\subset\R^3$ be such that $y_\eps\to y$ for some $y\in\R^3$
and let $\{w_\eps\}\subset E^+$ be such that $w_\eps\weakto w$
for some $w\in E^+$ then, up to a subsequence,
\[
\Phi^{red}_{\eps,y_\eps}(w_\eps)-\Phi^{red}_{\eps,y_\eps}(w_\eps-w)-\Phi^{red}_y(w)=o_\eps(1)
\]
and
\[
D\Phi^{red}_{\eps,y_\eps}(w_\eps)[\va]-D\Phi^{red}_{\eps,y_\eps}(w_\eps-w)[\va]-D\Phi_y(w)[\va]=o_\eps(1)\|\va\|
\]
uniformly for $\va\in E^+$ as $\eps\to0$.
\end{itemize}
%Moreover, for $\nu_1,\nu_2\in(-a,a)$ and the functionals $\Phi_{\nu_1},\Phi_{\nu_2}$ defined
%in \eqref{limit problem functional}, we have that
%$\|\msj_{\nu_1}(w)-\msj_{\nu_2}(w)\|=O(|\nu_1-\nu_2|)$ for
%each $w\in E^+$.
\end{Prop}

\section{The min-max scheme}\label{sec:min-max scheme}

In this section, we will prove the existence of solutions to the truncated problem \eqref{truncated dirac}
and, by virtue of Lemma \ref{critical attained},
we will restrict ourselves in the barrier
$0\leq\ka<\bar\ka$ where $\bar\ka$ is define in \eqref{eq:def-ka}.
We would like to emphasize that such a choice of $\bar\ka$
can be interpreted as we choose $c_0=\ga(J_{\vecmu_V})$ in Proposition \ref{PS condition}.
With all these notations, for such choice of $\ka$, we can fix
the constant $\de_0>0$ properly small so that the Palais-Smale condition
holds automatically in the energy range $\Phi_\eps\leq \ga(J_{\vecmu_V})$.

To begin with, let us mention that, under our hypotheses on $V$, there always exists a vector space
$X\subset\R^3$ such that:
\begin{itemize}
\item[$(a)$] $V|_X$ has a strict local maximum at $0$;

\item[$(b)$] $V|_{X^\bot}$ has a strict local minimum at $0$.
\end{itemize}
In fact, in case $(V1)$, $X=\R^3$ if $0$ is local maximum or
$X=\{0\}$ if $0$ is local minimum, whereas,
in case $(V2)$, $X$ is the space spanned by eigenvectors
associated to negative eigenvalues of $D^2 V(0)$. Let $P_X :\R^3\to X$ be the orthogonal projection (in the case
$X=\{0\}$, $P_X$ is simply the trivial projection).

In the next, solutions of \eqref{truncated dirac} will be obtained as critical
points of $\Phi_\eps$, and a key ingredient for the construction of a min-max scheme is using the
reduction couple $(\Phi^{red}_\eps, h_\eps)$. However,
due to the lack of information on the exact behavior of the reduction map
$h_\eps: E^+\to E^-$, it seems hopeless  to make a "path of least energy
spikes" by proper scaling as was employed in \cite{DPR, JeanjeanTanaka}.

Recalling $\nu_0=V(0)$, let us focus on functions in the subspace $E^+$. Denoted by
$B_0:=B(0,R_0)$ for some $R_0<R_1$. Let us choose a minimizer $U\in\cm(I_{\nu_0})$
for $\ga(I_{\nu_0})$ and consider the path $p_\eps: B_0^\eps\to\cm^+(\Phi_\eps)$
defined as
\[
p_\eps(\xi)(x)=t_{\xi,\eps} U^+(x-\xi), \quad x\in\R^3,
\]
where $\cm^+(\Phi_\eps)=\big\{ w\in E^+\setminus\{0\}:\, D\Phi^{red}_\eps(w)[w]=0 \big\}$ and $t_{\xi,\eps}$
is the unique $t>0$ such that
\[
t_{\xi,\eps}U^+(\cdot -\xi)\in\cm^+(\Phi_\eps).
\]
%
%
%Inspired by \cite{DPR}, here we define the topological cone in $E$
%\[
%\cc_\eps\equiv\big\{ tU_\xi:\, t\in[0,T_0],\, \xi\in \ov{ B_0^\eps}\cap X \big\},
%\]
%where $U\in\widetilde\msl_{\nu_0}$ is fixed, $U_\xi(x)=U(x-\xi)$ and $T_0>0$ (large enough)
%at which $\Phi_{\nu_0}(T_0 U)<0$.
We also define a family of deformations on $\cm^+(\Phi_\eps)$
\[
\Ga_\eps\equiv\big\{ \va:\cm^+(\Phi_\eps)\to \cm^+(\Phi_\eps) \text{ homeomorphism}:\,
\va(p_\eps(\xi))=p_\eps(\xi) \text{ if } \xi\in\pa B_0^\eps\cap X\big\}.
\]
Then we define the min-max level
\begin{\equ}\label{min-max level}
\ga_\eps:=\inf_{\va\in\Ga_\eps}\max_{\xi\in \ov{B_0^\eps}\cap X}\Phi^{red}_\eps(\va(p_\eps(\xi))).
\end{\equ}
We point out here that, in the case $X=\{0\}$,
$\ga_\eps=\ga(\Phi_\eps)=\inf_{\cm^+(\Phi_\eps)}\Phi^{red}_\eps$.
A technical point we would like to emphasize, which constitutes a crucial difference with min-max
quantity defined in \cite{DPR}, is the fact that the elements $p_\eps(\xi)+h_\eps(p_\eps(\xi))$
do not resemble a least energy solution of $I_\nu$ since
not much is known about the map $h_\eps: E^+\to E^-$.

\begin{Prop}\label{boundary energy estimate}
There exist $\eps_0,\,\de>0$ such that for any $\eps\in(0,\eps_0)$
\[
\Phi^{red}_\eps(p_\eps(\cdot))\big|_{\pa B_0^\eps\cap X}\leq \ga(J_{\nu_0})-\de.
\]
\end{Prop}
\begin{proof}
To simplify notation, we use subscript "$\xi$" to indicate the coordinate translation
of a function $u\in E$, that is, $u_\xi(x)=u(x-\xi)$. Then, on
a fixed bounded interval $t\in[0,T_0]$ with some $T_0$ large, we have
\[
\aligned
\Phi^{red}_\eps(tW_\xi)&\leq\frac12\big( \|tW_\xi\|^2-\|h_\eps(tW_\xi)\|^2 \big)
+\frac12\int_{\R^3}V_\eps(x)| tW_\xi+h_\eps(tW_\xi)|^2dx \\
&\qquad -\int_{B_1^\eps}F(|tW_\xi+h_\eps(tW_\xi)|)dx.
\endaligned
\]
Let us first remark that there exists $\si>0$ such that
$V(\xi)\leq \nu_0-\si$ for all $\xi\in\pa B_0^\eps\cap X$.
Since $t\in[0,T_0]$ is bounded and $R_0<R_1$, by $(1)$ in Proposition \ref{h-vr converges},
$h_\eps(t W_\xi)=h_\eps(t W)_\xi\to \msj_{V(\xi)}(tW)$ uniformly
in $t$ as $\eps\to0$. Thus, we deduce
\[
\Phi^{red}_\eps(tW_\xi)\leq J^{red}_{\nu_0-\si}(tW)+o_\eps(1)
\quad \forall \xi\in\pa B_0^\eps\cap X.
\]
Finally, since $J_{\nu_0-\si}<J_{\nu_0}$ strictly on compact subsets,
we have that
\[
\aligned
\max_{t>0}J^{red}_{\nu_0-\si}(tW)&=\max_{t>0}J_{\nu_0-\si}(tW+\msj_{\nu_0-\si}(tW))\\
&<\max_{t>0}J_{\nu_0}(tW+\msj_{\nu_0-\si}(tW))\\
&\leq\max_{t>0}J^{red}_{\nu_0}(tW)=\ga(J_{\nu_0}),
\endaligned
\]
which completes the proof.
\end{proof}

\begin{Prop}\label{gavr up estimate}
We have that
\[
\limsup_{\eps\to0}\ga_\eps\leq \ga(J_{\nu_0}).
\]
\end{Prop}
\begin{proof}
It suffices to show that
\begin{\equ}\label{XX}
\limsup_{\eps\to0}\max_{\xi\in \ov{B_0^\eps}\cap X}\Phi^{red}_\eps(p_\eps(\xi))\leq \ga(J_{\nu_0}).
\end{\equ}
In the following we take a sequence $\eps=\eps_n\to0$, but we drop the sub-index $n$ for the sake of clarity. For every $\eps$, there exists a maximum point $\xi_\eps\in B_0^\eps\cap X$ such that
\[
\max_{\xi_\eps\in \ov{B_0^\eps}\cap X}\Phi^{red}_\eps(p_\eps(\xi))=\Phi^{red}_\eps(p_\eps(\xi_\eps)).
\]
And we see that
\[
\aligned
\Phi^{red}_\eps(p_\eps(\xi_\eps))
&\leq\frac12\big( \|t_\eps W_{\xi_\eps}\|^2 -\|h_\eps(t_\eps W_{\xi_\eps})\|^2\big)
+\frac12\int_{\R^3}V_\eps(x)| t_\eps W_{\xi_\eps}+h_\eps(t_\eps W_{\xi_\eps})|^2dx\\
&\qquad -\int_{B_1^\eps}F(|t_\eps W_{\xi_\eps}+h_\eps(t_\eps W_{\xi_\eps})|)dx,
\endaligned
\]
where $t_\eps=t_{\xi_\eps,\eps}$. Since we have $\{t_\eps\}$ is bounded (up to a subsequence), we can assume that $t_\eps\to t_0$ and
$\eps\xi_\eps\to\xi_0\in \ov{B_0}\cap X$. Then we can conclude
that
\[
\aligned
\Phi^{red}_\eps(p_\eps(\xi_\eps))&\leq \frac12\big( \|t_0 W\|^2 -\|\msj_{V(\xi_0)}(t_0 W)\|^2\big)
+\frac{V(\xi_0)}2\int_{\R^3}| t_0 W+\msj_{V(\xi_0)}(t_0 W)|^2dx\\
&\qquad -\int_{\R^3}F(|t_0 W+\msj_{V(\xi_0)}(t_0 W)|)dx+o_\eps(1)\\
&= J_{V(\xi_0)}^{red}(t_0W)+o_\eps(1).
\endaligned
\]
Notice that $V(\xi_0)\leq\nu_0$, then
\[
J_{V(\xi_0)}^{red}(t_0W)\leq \max_{t>0}J_{\nu_0}^{red}(tW)=\ga(J_{\nu_0}),
\]
and hence \eqref{XX} holds.
\end{proof}

In the next, we will show that $\ga_\eps$ is a critical value of $\Phi_\eps$.
Motivated by \cite{DPR, Del-Pino3}, we are going to give an estimate from
below on $\ga_\eps$ and show that $\ga_\eps\geq \ga(J_{\nu_0})+o_\eps(1)$. And
in order to do so, we need to compare $\ga_\eps$ with another auxiliary
minimization value.
Firstly, set $B_3=B(0,3R_1)$ the open ball of radius $3R_1$ and
$\zeta:\R^3\to\R^3$ be a cut-off function
\begin{\equ}\label{zeta}
\zeta(x)=\left\{
\aligned
&x \quad &\text{if } |x|<3R_1,\\
&{3R_1 x}/{|x|} \quad &\text{if } |x|\ge3R_1,
\endaligned\right.
\end{\equ}
%$\zeta_\eps(x)=\zeta(\eps x)$
and let $Q_\eps:\R^3\to X$ be defined as
$Q_\eps(x)=P_X(\zeta(\eps x))$.
%, where
%$\chi_{B_3^\eps}$ is the characteristic function related to $B_3^\eps$.
Then, let us define the barycenter type functional $\cb_\eps:E\setminus\{0\}\to\R$,
\[
\cb_\eps(u)=\frac{\int_{\R^3}Q_\eps(x)|u|^\theta dx}{\int_{\R^3}|u|^\theta dx},
\quad \text{for } u\in E\setminus\{0\}
\]
where $\theta\in(2,3)$ is the constant in $(f2)$.
Recall that $(\Phi^{red}_\eps, h_\eps)$ is the reduction couple for $\Phi_\eps$ and
$\cm^+(\Phi_\eps)=\big\{w\in E^+\setminus\{0\}:\, D\Phi^{red}_\eps(w)[w]=0\big\}$, let us
consider the following subset of functions in $\cm^+(\Phi_\eps)$:
\[
\widetilde{\cm^+}(\Phi_\eps)=\big\{ w\in \cm^+(\Phi_\eps):\, \cb_\eps(w)=0 \big\}.
\]
We also define the corresponding auxiliary minimization
\begin{\equ}\label{b-eps}
b_\eps\equiv \inf_{w\in\widetilde{\cm^+}(\Phi_\eps)}\Phi^{red}_\eps(w).
\end{\equ}
When $X$ is trivial, i.e. $X=\{0\}$, we have
$\widetilde{\cm^+}(\Phi_\eps)=\cm^+(\Phi_\eps)$ and then $b_\eps=\ga_\eps$.

The next lemma shows that $b_\eps$ is well-defined in general.
\begin{Lem}\label{bvr is well defined}
There exists $\eps_0,\, \varrho>0$ such that for $\eps\in(0,\eps_0)$,
\[
\ga_\eps\geq b_\eps\geq\varrho.
\]
\end{Lem}

Technically, the crucial difference with the barycenter quantity
defined in \cite{DPR, Del-Pino3} is that the integrations in
$\cb_\eps$ are taken over the whole space $\R^3$.
The reason is twofold: firstly, the orthogonal projections associated to the decomposition $E=E^+\op E^-$
are of convolution type with some tempered distributions $\rho^\pm$ (see an abstract result in
\cite{Grafakos} for operators that commutes with translations), and thus, making the choice of
compact-supported functions in $E^\pm$ by simply multiplying smooth cut-off functions
would be in our situation hopeless since the convolution with
$\rho^\pm$ do not commute with the multiplication in general.
Secondly, the barycenter of an element $w\in E^+$ does not exhibit the location of the mass
of those $u\in E$ with $u^+=w$. Therefore, it is not enough if we only consider the barycenter
integrations over a bounded domain as was introduced in \cite{DPR, Del-Pino3}.

%It is clear that the auxiliary minimization $b_\eps$ gives rise to a certain
%Lagrange multiplier $\la_\eps\in X$, and some estimates should be made on it.
%Remark that Proposition \ref{invariant under derivatives} shows the orthogonal decomposition
%$E=E^+\op E^-$ remains invariant under the operations of derivatives.
%Our construction in \eqref{zeta} seems very natural, as we will see in the next
%sections, because of the sign of directional derivative of $\la_\eps\cdot \zeta(\cdot)$
%along the unit vector $\la_\eps/|\la_\eps|$ inside $B_3$ can be transmitted
%outside to $B_3^c$. Notice that $\zeta$ is uniformly bounded on $\R^3$,
%it turns out that $\cb_\eps$ behaves quite similar as the barycenter
%defined in \cite{DPR, Del-Pino3} counting only with very rough information
%on the reduced functional $I_\eps$ on $E^+$.

%Thanks to this, the presence of saddle point on $V$
%may be seen inducing a
%to the reduced functionals $I_\eps$ counting with very rough information.

\begin{proof}[Proof of Lemma \ref{bvr is well defined}]
Since $b_\eps\geq\varrho$ follows directly from $(f1)-(f3)$ for some $\varrho>0$,
we only need to prove that $\ga_\eps\geq b_\eps$ for all small $\eps$.

Motivated by \cite{DPR}, let us take an arbitrary $\va\in\Ga_\eps$.
We define $\psi_\eps: \ov{B_0}\cap X\to X$
as
\[
\psi_\eps(\xi)=\cb_\eps\big(\va(p_\eps(\xi/\eps))\big).
\]
We point out here that, by the definition of $\Ga_\eps$, $\va(p_\eps(\xi/\eps))\neq0$
for all $\xi\in\ov{B_0}\cap X$, and so $\psi_\eps$ is well defined.

For $\xi\in\pa B_0\cap X$, it can be seen from the definition of $\cb_\eps$ that
\[
\psi_\eps(\xi)=\xi+o_\eps(1)  \text{ uniformly in } \xi\in\pa B_0\cap X, \text{ as } \eps\to0.
\]
Therefore we can choose $\eps_0$ small enough (independent of $\va$) so that,
for all $\eps\in(0,\eps_0)$,
\[
\text{deg}(\psi_\eps,\, B_0\cap X,\,0)=\text{deg}(id,\, B_0\cap X,\,0)=1.
\]
Then we can conclude that for every $\eps$, there exists $\xi_\eps\in B_0\cap X$
such that $\psi_\eps(\xi_\eps)=0$.

Therefore, since $\xi_\eps/\eps\in \ov{B_0^\eps}\cap X$, there follows
\[
\max_{\xi\in \ov{B_0^\eps}\cap X}\Phi^{red}_\eps(\va(p_\eps(\xi)))\geq \Phi^{red}_\eps(\va(p_\eps(\xi_\eps/\eps)))\geq b_\eps,
\]
which concludes the proof.

\end{proof}

\begin{Prop}\label{key prop}
We have that
\[
\liminf_{\eps\to0}b_\eps\geq\ga(J_{\nu_0}).
\]
\end{Prop}

The proof of this proposition contains the main difficulties of the paper. It will be presented in the next section.
Assuming the conclusion for the moment, jointly with
Proposition \ref{gavr up estimate}, we can obtain the following

\begin{Prop}\label{gavr limit}
We have that
\[
  \lim_{\eps\to0}\ga_\eps=\ga(J_{\nu_0}).
\]
\end{Prop}

From Proposition \ref{boundary energy estimate} and \ref{gavr limit},
we can get $\ga_\eps>\Phi^{red}_\eps(p_\eps(\cdot))
\big|_{\pa B_0^\eps\cap X}$ for
all small $\eps>0$. Recall that we have restricted $\ka\in[0,\bar\ka)$,
it follows that $\ka^2\cdot\ga(J_{\nu_0})<
\Big(\frac{a^2-|V|_\infty^2}{a^2}\Big)^{\frac32}\frac{S^{\frac32}}6$
which guarantees the compactness. Thus, by
Proposition \ref{PS condition}, we easily obtain

\begin{Thm}\label{existence part}
There exists $\eps_0>0$ such that for $\eps\in(0,\eps_0)$ there exists a solution $z_\eps$ of the problem \eqref{truncated dirac}. Moreover,
$\Phi^{red}_\eps(z_\eps^+)=\Phi_\eps(z_\eps)=\ga_\eps$.
\end{Thm}

\section{Proof of Proposition \ref{key prop}}\label{sec:Proof of key prop}
The proof will be divided into several parts. As a first step, we prove the existence of a minimizer $u_\eps$ to be auxiliary problem \eqref{b-eps}.

\begin{Lem}\label{STEP1}
There exists $\eps_0>0$  such that for any $\eps\in(0,\eps_0)$, there exist
$u_\eps\in E\setminus\{0\}$ with $\cb_\eps(u_\eps^+)=0$ and $\la_\eps\in X$
such that
\begin{\equ}\label{constraint equation}
-i\al\cdot\nabla u_\eps + a\be u_\eps + V_\eps(x)u_\eps = g_\eps(x,|u_\eps|)u_\eps +
\big(\la_\eps\cdot Q_\eps(x) |u_\eps^+|^{\theta-2}u_\eps^+\big)^+
\end{\equ}
and
\[
\Phi_\eps(u_\eps)=b_\eps.
\]
Moreover, the sequence $\{u_\eps\}$ is bounded in $E$.
\end{Lem}
\begin{proof}
We sketch the proof as follows. For $\eps>0$ fixed, by the
Ekeland variational principle, there exists a sequence
$\{w_n\}\subset\widetilde{\cm^+}(\Phi_\eps)$ which is a constrained
$(PS)$-sequence for $\Phi^{red}_\eps$ at level $b_\eps$, moreover, it
can be deduced that there
exists $\{\la_n\}\subset X$ such that
\begin{\equ}\label{constraint1}
\Phi^{red}_\eps(w_n)\to b_\eps, \quad \text{as } n\to\infty,
\end{\equ}
\begin{\equ}\label{constraint2}
D\Phi^{red}_\eps(w_n)-\frac{(\la_n\cdot Q_\eps(x)|w_n|^{\theta-2}w_n)^+}{|w_n|_\theta^\theta}
\to0, \quad \text{as } n\to\infty.
\end{\equ}
Now, let us set $u_n=w_n+h_\eps(w_n)$.
Since $\cb_\eps(u_n^+)=\cb_\eps(w_n)=0$,
by \eqref{constraint1} and \eqref{constraint2},
repeating the arguments of Lemma \ref{boundedness}, we get that
$\{u_n\}$ is bounded in $E$ (uniformly with respect to $\eps$) and,
therefore, up to a subsequence, it converges weakly to some $u_\eps\in E$.
Since we have assumed $0\leq \ka<\bar\ka$, it follows that
$b_\eps\leq \ga_\eps\leq \ga(J_{\nu_0})+o_\eps(1)\leq \ga(J_{\vecmu_V})$ for small $\eps$. By
Proposition \ref{PS condition}, $\{u_n\}$ converges strongly in $E$, i.e.
$u_n\to u_\eps$ as $n\to\infty$. Note that $u_\eps\neq0$, $\liminf_{\eps\to0}b_\eps>0$, also the
sequence $\la_n$ is bounded, we have $u_\eps$ is the desired minimizer and this concludes the proof.
\end{proof}

\begin{Lem}\label{non-vanishing uvr}
We have that $u_\eps^+\chi_{B_2^\eps}$ is non-vanishing.
\end{Lem}
\begin{proof}
We only consider the case $\ka>0$ since it is much easier when $\ka=0$.
To the contrary, we assume that $u_\eps^+\chi_{B_2^\eps}$ vanishes. Then
we have $u_\eps^+\chi_{B_2^\eps}\to0$ in $L^q$ for all $q\in(2,3)$. At this
point we first claim that
\begin{equation}\label{claim2}
  u_\eps^+\chi_{B_2^\eps}\not\to0 \ \text{ in } L^3.
\end{equation}
\noindent
Accepting this fact for the moment, let us consider the function
\[
t\mapsto \Phi_\eps(t u_\eps^+)
\]
and denote $t_\eps>0$ the unique maximum point which realizes its maximum. Then $\{t_\eps\}$ is
bounded. Set $z_\eps=t_\eps u_\eps^+\in E^+$, we have that $D\Phi_\eps(z_\eps)[z_\eps]=0$ and hence
\[
\aligned
\|z_\eps\|^2+\int_{\R^3}V_\eps(x)|z_\eps|^2dx =\int_{\R^3}
(1-\chi_\eps(x))\tilde f(|z_\eps|)|z_\eps|^2dx
+\ka\int_{\R^3}\chi_\eps(x)|z_\eps|^3dx+o_\eps(1).
\endaligned
\]
Since $u_\eps^+\chi_{B_2^\eps}\not\to0$ in $L^3$, similarly as that was argued in
Proposition \ref{PS condition}, we soon have that
\[
\ka^3\int_{\R^3}\chi_\eps(x)|z_\eps|^3dx+o_\eps(1)\geq
\Big( \frac{a^2-(|V|_\infty+\de_0)^2}{a^2} \Big)^{\frac32}S^{\frac32}.
\]
And hence, thanks to our choice of $\ka\in(0,\bar\ka)$,
we get
\[
\aligned
\ka^2\Phi_\eps(z_\eps)&=\ka^2\Big( \Phi_\eps(z_\eps)-\frac12\Phi_\eps'(z_\eps)[z_\eps] \Big)\\
&\geq\Big( \frac{a^2-(|V|_\infty+\de_0)^2}{a^2} \Big)^{\frac32}\frac{S^{\frac32}}6
+o_\eps(1)\\
&>\ka^2 \ga(J_{\vecmu_V}).
\endaligned
\]
Therefore, we have that
\[
\ga(J_{\vecmu_V})<\Phi_\eps(z_\eps)\leq \max_{t>0}\Phi^{red}_\eps(tu_\eps^+)=b_\eps\leq \ga(J_{\nu_0})
\quad \text{as } \eps \to0
\]
which is impossible due to Lemma \ref{critical attained}.

\medskip

Now, it remains to show \eqref{claim2} is valid. Indeed, it follows from Lemma \ref{STEP1} that, for some $C>0$,
\[
\aligned
b_\eps&=\Phi_\eps(u_\eps)=\max_{t>0}\Phi^{red}_\eps(tu_\eps^+)
\geq\max_{t>0}\Phi_\eps(tu_\eps^+)\\
&\geq\max_{t>0}\Big[\,\frac{t^2}2\Big(1-\frac{|V|_\infty+\de_0}{a}
\Big)\|u_\eps^+\|^2
-C\ka t^3\int_{B_2^\eps}|u_\eps^+|^3dx\Big].
\endaligned
\]
Then, if $u_\eps^+\chi_{B_2^\eps}\to0$ in $L^3$ as $\eps\to0$,
we can choose $T_0>0$
(independent of $\eps$) large enough such that
$\Phi_\eps(T_0 u_\eps^+)>2\ga(J_{\vecmu_V})$ for all small $\eps>0$, and we
soon have that
\[
\liminf_{\eps\to0}b_\eps\geq\liminf_{\eps\to0}\Phi_\eps(T_0 u_\eps^+)> \ga(J_{\vecmu_V})
\]
which is absurd.
\end{proof}

\begin{Lem}
We have that $\{\la_\eps\}\subset X$ is bounded.
\end{Lem}
\begin{proof}
Let us assume that $\la_\eps\neq0$, otherwise we are done. In the sequel, let us set
$\tilde\la_\eps=\la_\eps/|\la_\eps|$. %Choose $\phi_\eps:\R^3\to[0,1]$ be a
%smooth function such that
%\[
%\phi_\eps(x)=\left\{
%\aligned
%&1 &\quad &\text{in } B_2^\eps ,\\
%&0 &\quad &\text{in } (B_3^\eps)^c,
%\endaligned\right.
%\]
%with $|\nabla\phi_\eps|=O(\eps)$.
By elliptic regularity arguments we have that $u_\eps\in \cap_{q\geq2}W^{1,q}(\R^3,\C^4)$
and then, jointly with Proposition \ref{invariant under derivatives},
we are allowed to multiply \eqref{constraint equation}
by $\pa_{\tilde\la_\eps}u_\eps$. Then, we have
\begin{\equ}\label{lm1}
\aligned
&\real\int_{\R^3}\Big(
-i\al\cdot\nabla u_\eps+a\be u_\eps + V_\eps(x) u_\eps - g_\eps(x,|u_\eps|)u_\eps \Big)
\cdot \ov{\pa_{\tilde\la_\eps}u_\eps} \,dx \\
&=\real\int_{\R^3}\la_\eps\cdot Q_\eps(x)|u_\eps^+|^{\theta-2}
u_\eps^+\cdot
\ov{\pa_{\tilde\la_\eps}u_\eps^+} \,dx.%\\
%&=\real\int_{\R^3}\big( \la_\eps\cdot Q_\eps(x)u_\eps^+ \big)^+\cdot
%\ov{(\phi_\eps\pa_{\tilde\la_\eps}u_\eps)} \,dx,
\endaligned
\end{\equ}
%where the last equality follows from \eqref{l2dec}.
Now, let us evaluate each term
of the previous equality. We get
\[
\aligned
0=\real\int_{\R^3}\pa_{\tilde\la_\eps}\big[
(-i\al\cdot\nabla u_\eps)\cdot \ov{u_\eps} \,\big]dx
&=2\real\int_{\R^3}(-i\al\cdot\nabla u_\eps)\cdot\ov{\pa_{\tilde\la_\eps}u_\eps}dx%\\
%&\qquad +\real\int_{\R^3}(-i\al\cdot(\nabla\phi_\eps)u_\eps)\cdot\ov{(\pa_{\tilde\la_\eps}u_\eps)}dx
\endaligned
\]
and so
\begin{\equ}\label{lm2}
\real\int_{\R^3}(-i\al\cdot\nabla u_\eps)\cdot\ov{\pa_{\tilde\la_\eps}u_\eps}dx
=0
%-\frac12\real\int_{B_3^\eps}(-i\al\cdot(\nabla\phi_\eps)u_\eps)\cdot\ov{(\pa_{\tilde\la_\eps}u_\eps)}dx=O(\eps).
\end{\equ}
Analogously, we have
\[
\aligned
0&=\int_{\R^3}\pa_{\tilde\la_\eps}\big[ V_\eps(x) |u_\eps|^2\big]dx\\
&=\eps\int_{\R^3}\pa_{\tilde\la_\eps}V(\eps x) |u_\eps|^2 dx
+2\real\int_{\R^3}V_\eps(x)u_\eps\cdot\ov{\pa_{\tilde\la_\eps}u_\eps}dx
%+\int_{\R^3}V_\eps(x)|u_\eps|^2\pa_{\tilde\la_\eps}\phi_\eps dx
\endaligned
\]
and so
\begin{\equ}\label{lm3}
\aligned
\real\int_{\R^3}V_\eps(x)u_\eps\cdot\ov{\pa_{\tilde\la_\eps}u_\eps}dx&=
-\frac\eps2\int_{\R^3}\pa_{\tilde\la_\eps}V(\eps x) |u_\eps|^2 dx
%-\frac12\int_{\R^3}V_\eps(x)|u_\eps|^2\pa_{\tilde\la_\eps}\phi_\eps dx \\
= O(\eps).
\endaligned
\end{\equ}
It also follows that
\begin{\equ}\label{lm4}
\real\int_{\R^3}a\be u_\eps\cdot\ov{\pa_{\tilde\la_\eps}u_\eps}dx=0.
\end{\equ}
For the nonlinear part, let us recall the definition of $G_\eps$,
\[
\pa_{\tilde\la_\eps}G_\eps(x,|u_\eps|)=\eps\pa_{\tilde\la_\eps}\chi(\eps x)\big(
F(|u_\eps|)-\tilde F(|u_\eps|)\big) +\real \,g_\eps(x,|u_\eps|)u_\eps\cdot\ov{\pa_{\tilde\la_\eps}u_\eps},
\]
then we have
\[
\aligned
0&=\int_{\R^3}\pa_{\tilde\la_\eps}\big[ G_\eps(x,|u_\eps|) \big]dx\\
&=\eps\int_{\R^3}\big(F(|u_\eps|)-\tilde F(|u_\eps|)\big)\big( \pa_{\tilde\la_\eps}\chi(\eps x) \big)
dx + \real\int_{\R^3} g_\eps(x,|u_\eps|)u_\eps\cdot\ov{\pa_{\tilde\la_\eps}u_\eps}dx%\\
%&\qquad + \int_{\R^3}F_\eps(x,|u_\eps|)\pa_{\tilde\la_\eps}\phi_\eps dx
\endaligned
\]
and it follows that
\begin{\equ}\label{lm5}
\real\int_{\R^3} g_\eps(x,|u_\eps|)u_\eps\cdot\ov{\pa_{\tilde\la_\eps}u_\eps}dx
=O(\eps).
\end{\equ}
Finally
\[
\aligned
0&=\int_{\R^3}\pa_{\tilde\la_\eps}\big[ \la_\eps\cdot Q_\eps(x)|u_\eps^+|^\theta \big]dx\\
&=\eps|\la_\eps|\int_{B_3^\eps}|u_\eps^+|^\theta dx+\eps |\la_\eps|\int_{\R^3\setminus B_3^\eps}
\frac{R_3}{\eps|x|}
\Big[ 1- \frac{(\la_\eps\cdot x)^2}{|\la_\eps|^2|x|^2}  \Big] |u_\eps^+|^\theta dx\\
&\qquad
+\theta\real\int_{\R^3}\la_\eps\cdot Q_\eps(x)|u_\eps^+|^{\theta-2}
u_\eps^+\cdot\ov{\pa_{\tilde\la_\eps}u_\eps^+}dx
\endaligned
\]
Observe that $0\leq \pa_{\tilde\la_\eps}\la_\eps\cdot Q_\eps(x)\leq \eps|\la_\eps|$ for all
$x\in\R^3\setminus B_3^\eps$; this is the key point of our estimates.
And hence
\begin{\equ}\label{lm6}
\aligned
\real\int_{\R^3}\la_\eps\cdot Q_\eps(x)|u_\eps^+|^{\theta-2}
u_\eps^+\cdot\ov{\pa_{\tilde\la_\eps}u_\eps^+}dx
&=-\frac{\eps|\la_\eps|}\theta\int_{B_3^\eps}|u_\eps^+|^2dx  \\
&\qquad
-\frac{\eps |\la_\eps|}\theta\int_{\R^3\setminus B_3^\eps}\frac{R_3}{\eps|x|}
\Big[ 1- \frac{(\la_\eps\cdot x)^2}{|\la_\eps|^2|x|^2}  \Big] |u_\eps^+|^2dx.
\endaligned
\end{\equ}
By \eqref{lm1}-\eqref{lm6} and Lemma \ref{non-vanishing uvr}, we conclude
the boundedness of $\la_\eps\in X$.
\end{proof}

%The next lemma gives a uniformly $W^{1,q}_{loc}(\R^3,\C^4)$ estimate for
%$u_\eps$ as $\eps\to0$.
%
%\begin{Lem}\label{W1qloc}
%We have that $u_\eps\in \cap_{q\geq2}W^{1,q}_{loc}(\R^3,\C^4)\cap L^\infty(\R^3,\C^4)$.
%\end{Lem}

In what follows, we consider a sequence $\eps_k\to0$ and assume that
$\la_{\eps_k}\to\bar\la\in X$. For simplicity, we still denote $\eps_k$ by $\eps$.
For a small $\de>0$, let us define
\[
H_\eps=\big\{ x\in\R^3:\, \bar\la\cdot Q_\eps(x)\leq \de \big\}.
\]

The next proposition gives a complete description of $u_\eps$ as $\eps\to0$. We recall the notations $B_2=B(0,2R_1)$ and $B_3=B(0,3R_1)$.

\begin{Prop}\label{u-vr converges}
Passing to a subsequence if necessary, there exist
 $y_\eps^1\in H_\eps$, $y_1\in B_2$ and $u_1\in E\setminus\{0\}$
with
\[
-i\al\cdot\nabla u_1+a\be u_1 + V(y_1)u_1 = g(y_1,|u_1|)u_1,
\]
such that $\bar\la\cdot y_1=0$ and
\[
\eps y_\eps^1\to y_1, \quad
\|u_\eps-u_1(\cdot-y_\eps^1)\|\to0 \quad \text{as } \eps\to0.
\]
\end{Prop}
%\todo[inline]{This lemma is the main reason that we choose the method of
%"reduction couple", for details please see discussion file 1801220-2}
\begin{proof}
We divide the proof into different steps:

\medskip
\textbf{Step 1.} $u_\eps^+|_{H_\eps}\not\to0$ in the $L^2$-norm and $L^\theta$-norm.

Let us first show that $u_\eps^+\not\to 0$ in $L^\theta(H_\eps)$. Suppose contrarily that
\[
\int_{H_\eps}|u_\eps^+|^{\theta}dx\to0, \quad \text{as } \eps\to0.
\]
Since $\cb_\eps(u_\eps^+)=0$ and $\bar\la\in X$, we have
\[
\aligned
0=\int_{\R^3}\bar\la\cdot Q_\eps(x)|u_\eps^+|^\theta dx&=
\int_{H_\eps}\bar\la\cdot Q_\eps(x)|u_\eps^+|^\theta dx+
\int_{H_\eps^c}\bar\la\cdot Q_\eps(x)|u_\eps^+|^\theta dx\\
&\geq \de\int_{H_\eps^c}|u_\eps^+|^\theta dx
+\int_{H_\eps}\bar\la\cdot Q_\eps(x)|u_\eps^+|^\theta dx.
\endaligned
\]
Therefore
\[
\de\int_{H_\eps^c}|u_\eps^+|^\theta dx\leq \Big| \int_{H_\eps}\bar\la\cdot Q_\eps(x)|u_\eps^+|^\theta dx\Big|
\leq |\bar\la| R_3\int_{H_\eps}|u_\eps^+|^{\theta}dx
\]
and so
\[
\int_{H_\eps^c}|u_\eps^+|^\theta dx\to 0, \quad \text{as } \eps\to0.
\]
Then we get $u_\eps^+\to0$ in $L^\theta$ which is a contradiction
with Lemma \ref{non-vanishing uvr}. Now, by the boundedness of
$\{u_\eps\}$ in $E$ and so in $L^3$, we can conclude by
interpolation: for a suitable $\mu\in(0,1)$
\[
0<c\leq\|u_\eps^+\|_{L^{\theta}(H_\eps)}\leq\|u_\eps^+\|_{L^2(H_\eps)}^\mu
\|u_\eps^+\|_{L^3(H_\eps)}^{1-\mu}\leq C \|u_\eps^+\|_{L^2(H_\eps)}^\mu.
\]

\medskip
\textbf{Step 2.} Passing to be limit by concentration-compactness.

By Step 1, we can conclude that $\{u_\eps^+|_{H_\eps}\}$ is non-vanishing.
And hence, by concentration-compactness arguments (see \cite{Lions}),
there exist $y_\eps^1\in H_\eps$ and $r>0$ such that
\[
\int_{B(y_\eps^1,r)\cap H_\eps} |u_\eps^+|^2\geq c>0.
\]
Therefore there exits $u_1\in E\setminus\{0\}$ such that
$v_\eps^1=u_\eps(\cdot+y_\eps^1)\weakto u_1$ in $E$.
\begin{claim}\label{y1 bdd R2}
$\{\eps y_\eps^1\}$  is bounded and, up to a subsequence, $\eps y_\eps^1\to y_1\in B_2$ as $\eps\to0$.
\end{claim}

\medskip
To see this, let us assume that $\eps y_\eps^1\not\in B_2$ and
$\dist(\eps y_\eps^1, \pa B_2)/\eps\to\infty$.
Observe that $v_\eps^1$ solves the equation
\[
-i\al\cdot\nabla v_\eps^1+a\be v_\eps^1+ V(\eps x + \eps y_\eps^1) v_\eps^1
=g(\eps x+ \eps y_\eps^1,|v_\eps^1|)v_\eps^1+
\big( \la_\eps\cdot Q_\eps(x+y_\eps^1) |v_\eps^{1+}|^{\theta-2}v_\eps^{1+}\big)^+,
\]
and if we assume that $V(\eps y_\eps^1)\to \nu_1$ as $\eps\to0$ (passing to
a subsequence), we have that
$u_1$ is a weak solution of
\begin{\equ}\label{limit equ 1}
-i\al\cdot\nabla u+a\be u +\nu_1 u =\tilde f(|u|)u +
\big(\, \bar\la\cdot\tilde y_1 |u^+|^{\theta-2}u^+ \big)^+
\end{\equ}
where $\tilde y_1\in  B_3$ is given by
\[
\tilde y_1=\left\{
\aligned
&\lim_{\eps\to0}\eps y_\eps^1 \quad & \text{if } \eps y_\eps^1\in B_3,\\
&\lim_{\eps\to0}\frac{3R_1 y_\eps^1}{|y_\eps^1|} \quad & \text{if } \eps y_\eps^1\in B_3^c.
\endaligned\right.
\]
Since $y_\eps^1\in H_\eps$, we have that $\bar\la\cdot \tilde y_1\leq \de$
and, by the definition of $\tilde f$,
we easily get that $\bar\la\cdot \tilde y_1>0$ (otherwise
$u_1^+$ should be $0$). Now we let $\widetilde\Phi_1: E\to\R$ denote the
associated energy functional for \eqref{limit equ 1}, that is
\[
\widetilde\Phi_1(u)=\frac12\big( \|u^+\|^2-\|u^-\|^2 \big) +\frac{\nu_1}2|u|_2^2
-\int_{\R^3}\tilde F(|u|)dx-\frac{\bar\la\cdot\tilde y_1}\theta\int_{\R^3}|u^+|^\theta dx.
\]
Remark that, for any $u\in E$ with $u^+\neq0$ and arbitrary $v\in E$, there holds
that
\[
\bar\la\cdot\tilde y_1\int_{\R^3}|u^+|^{\theta-2}|v^+|^2dx+
(\theta-2)\bar\la\cdot\tilde y_1\int_{\R^3} |u^+|^{\theta-2}
\Big( |u^+|+\frac{\real\, u^+\cdot \ov{v^+}}{|u^+|} \Big)^2dx>0.
\]
As a consequence of \cite[Theorem 5.1]{Ackermann} (see also
\cite[Lemma 4.6]{Ding-Xu-Trans}), we have that Theorem~\ref{thm:red-couple}
applies to the situation here. So, we can take
$(\widetilde \Phi_1^{red}, \tilde h_1)$
to be the reduction couple for $\widetilde\Phi_1$ and
let $\tilde \ga_1$ stand for the critical level realized by $u_1$,
we then have
\begin{eqnarray*}
\tilde\ga_1&=&\widetilde \Phi^{red}_1(u_1^+)=\max_{t>0}\widetilde\Phi_1^{red}(t u_1^+)
\geq\max_{t>0}\widetilde\Phi_1(t u_1^+)\\
&\geq&\max_{t>0} \frac{t^2}2\big( \|u_1^+\|^2- (|V|_\infty+\de_0)|u_1|_2^2\big)
-\frac{\bar\la\cdot\tilde y_1}\theta t^\theta \int_{\R^3}|u_1^+|^\theta dx\\
&\geq&\max_{t>0} \frac{t^2}2\big( \|u_1^+\|^2- (|V|_\infty+\de_0)|u_1|_2^2\big)
-\frac{\de}\theta t^\theta \int_{\R^3}|u_1^+|^\theta dx.
\end{eqnarray*}
Since $\|u_1\|\leq \|v_\eps^1\|=\|u_\eps\|<\infty$, we can conclude that $\tilde\ga_1>2\ga(J_{\nu_0})$
provided that $\de$ is fixed small enough. However, by Fatou's lemma, we get
\[
\aligned
\tilde\ga_1&=\widetilde\Phi_1(u_1)-\frac12D\widetilde\Phi_1(u_1)[u_1]
=\int_{\R^3}\frac12\tilde f(|u_1|)|u_1|^2-\tilde F(|u_1|)dx
+\big(\frac12-\frac1\theta\big)\bar\la\cdot\tilde y_1 |u_1^+|_\theta^\theta\\
&\leq\int_{\R^3}\frac12\tilde f(|u_1|)|u_1|^2-\tilde F(|u_1|)dx + O(\de)\\
&\leq O(\de)+\liminf_{\eps\to0}\int_{\R^3}
\frac12 g(\eps x + \eps y_\eps^1,|v_\eps^1|)|v_\eps^1|^2
- G(\eps x+\eps y_\eps^1, |v_\eps^1|)dx\\
&=O(\de)+\liminf_{\eps\to0}\Phi_\eps(u_\eps) \leq 2\ga(J_{\nu_0})
\endaligned
\]
which is impossible. This proves the claim.

\medskip

Now by Claim \ref{y1 bdd R2}, passing to the limit,
we have $u_1$ is a weak solution of
\[
-i\al\cdot\nabla u_1+a\be u_1 +V(y_1) u_1 = g(y_1,|u_1|)u_1
+\big(\, \bar\la\cdot y_1 |u_1^+|^{\theta-2}u_1^+ \big)^+,
\]
with $\eps y_\eps^1\to y_1\in B_2$ such that $\bar\la\cdot y_1\leq \de$
and there exits $\bar c>0$ such that
\[
\|u_\eps\|\geq\|u_1\|\geq \bar c>0.
\]
Let us define $z_{1,\eps}=u_\eps-u_1(\cdot-y_\eps^1)$. We consider two
possibilities: either $\|z_{1,\eps}^+\|\to0$ or not. In the first case the
proposition should be proved. In the second case, there are two
sub-cases: either $z_{1,\eps}^+|_{H_\eps}\to0$ in the $L^\theta$-norm or not.

\medskip
\textbf{Step 3.} Assume that $z_{1,\eps}^+|_{H_\eps}\not\to0$ in the $L^\theta$-norm.

In this case, we can repeat the previous argument to the sequence $\{z_{1,\eps}\}$
to obtain $y_\eps^2\in H_\eps$ such that
\[
\int_{B(y_\eps^2,r)\cap H_\eps} |z_{1,\eps}^+|^2\geq c>0.
\]
Therefore there exists $u_2\in E\setminus\{0\}$ such that
$v_\eps^2=z_{1,\eps}(\cdot+y_\eps^2)\weakto u_2$ in $E$.
Moreover, $|y_\eps^1-y_\eps^2|\to\infty$, $\eps y_\eps^2\to y_2\in B_2$, $\bar\la\cdot y_2\leq \de$
and
\[
-i\al\cdot\nabla u_2+a\be u_2 +V(y_2) u_2 = g(y_2,|u_2|)u_2
+\big(\, \bar\la\cdot y_2 |u_2^+|^{\theta-2}u_2^+ \big)^+,
\]
and $\|u_2\|\geq\bar c>0$. Also, it follows from the weak convergence,
\[
\|u_\eps\|^2\geq\|u_1\|^2+\|u_2\|^2.
\]
Let us set $z_{2,\eps}=u_\eps-u_1(\cdot-y_\eps^1)-u_2(\cdot-y_\eps^2)$. Suppose
that $\|z_{2,\eps}^+\|\not\to0$ and $z_{2,\eps}^+|_{H_\eps}\not\to0$ in
$L^\theta$, then we can argue again as above. And it is all clear that there exists $l\in\N$ such
that, after repeating the above argument for $l$ times,
we can get that $z_{l,\eps}^+|_{H_\eps}\to0$ in the $L^\theta$-norm.

\medskip
\textbf{Step 4.} $\|z_{l,\eps}^+\|\to0$ as $\eps\to0$.

To the contrary let us assume that $\|z_{l,\eps}^+\|\not\to0$.
Since $Q_\eps(\cdot)$ is bounded,
it follows from a standard argument that
\[
\aligned
&\real\int_{\R^3}\la_\eps\cdot Q_\eps(x)|u_\eps^+|^{\theta-2}u_\eps^+\cdot\ov{\va^+} dx\\
&\qquad =\sum_{j=1}^l\bar\la\cdot y_j \real\int_{\R^3}|u_j^+(\cdot-y_\eps^j)|^{\theta-2}
u_j^+(\cdot-y_\eps^j) \cdot\ov{\va^+}dx \\
&\qquad \qquad +\real\int_{\R^3}\la_\eps\cdot Q_\eps(x)
|z_{l,\eps}^+|^{\theta-2} z_{l,\eps}^+\cdot\ov{\va^+}dx+o_\eps(1)\|\va\|,
\endaligned
\]
uniformly for $\va\in E$ as $\eps\to0$ and, particularly,
\begin{\equ}\label{E0}
\int_{\R^3}\la_\eps\cdot Q_\eps(x)|u_\eps^+|^\theta dx=
\sum_{j=1}^l\bar\la\cdot y_j \int_{\R^3}|u_j^+|^\theta dx
+\int_{\R^3}\la_\eps\cdot Q_\eps(x)|z_{l,\eps}^+|^\theta dx+o_\eps(1).
\end{\equ}
Since $\cb_\eps(u_\eps^+)=0$, together with Proposition \ref{h-vr converges},
we can deduce from \eqref{E0} that
\begin{\equ}\label{E1}
\aligned
o_\eps(1)&=\|z_{l,\eps}^++h_\eps(z_{l,\eps}^+)\|^2+\real\int_{\R^3}
V_\eps(x)\big(z_{l,\eps}^++h_\eps(z_{l,\eps}^+)\big)\cdot
\ov{\big(z_{l,\eps}^+-h_\eps(z_{l,\eps}^+)\big)}dx \\
&\qquad -\real\int_{\R^3}g_\eps\big(x,|z_{l,\eps}^++h_\eps(z_{l,\eps}^+)|\big)
\big(z_{l,\eps}^++h_\eps(z_{l,\eps}^+)\big)\cdot
\ov{\big(z_{l,\eps}^+-h_\eps(z_{l,\eps}^+)\big)}dx \\
&\qquad -\int_{\R^3}\la_\eps\cdot Q_\eps(x)|z_{l,\eps}^+|^\theta dx.
\endaligned
\end{\equ}
Therefore, by $(f2)$ and Proposition \ref{lpdec}, we obtain
\[
\|z_{l,\eps}^++h_\eps(z_{l,\eps}^+)\|^2\leq C |z_{l,\eps}^++h_\eps(z_{l,\eps}^+)|_3^3
\leq C' \|z_{l,\eps}^++h_\eps(z_{l,\eps}^+)\|^3,
\]
for some $C,C'>0$ which implies there exists
$c>0$ such that $\|z_{l,\eps}^++h_\eps(z_{l,\eps}^+)\|\geq c$.
In what follows, for simplicity of notation, we denote
$\bar z_{l,\eps}=z_{l,\eps}^++h_\eps(z_{l,\eps}^+)$.
By \eqref{E1} again, and a similar argument as in the proof of Lemma \ref{boundedness}, we get that
\[
\aligned
\|\bar z_{l,\eps}\|^2&\leq C_\theta \Big(
\int_{\R^3}\chi_\eps(x)\big( f(|\bar z_{l,\eps}|)|\bar z_{l,\eps}|^2
-2F(|\bar z_{l,\eps}|) \big)dx \Big)^{\frac23}|\bar z_{l,\eps}^+-\bar z_{l,\eps}^-|_3\\
&\qquad +
C \int_{\R^3}\la_\eps\cdot Q_\eps(x)|\bar z_{l,\eps}^+|^\theta dx
+o_\eps(1) \\
&\leq C_\theta' \Big( 2\Phi^{red}_\eps(z_{l,\eps}^+) -D\Phi^{red}_\eps(z_{l,\eps}^+)[z_{l,\eps}^+]\Big)^{\frac23}\|\bar z_{l,\eps}\|
+
C \int_{\R^3}\la_\eps\cdot Q_\eps(x)|\bar z_{l,\eps}^+|^\theta dx
+o_\eps(1)
\endaligned
\]
for some $C,C_\theta,C_\theta'>0$. Remark that $\bar z_{l,\eps}^+
=z_{l,\eps}^+\to0$ in $L^\theta(H_\eps)$.
Then, it follows from $\|\bar z_{l,\eps}\|\geq c$
and $(f2)$
that there exists constant $c'>0$ (independent of $R_1$) such that
\begin{\equ}\label{E2}
\liminf_{\eps\to0} \Big(
\Phi^{red}_\eps(z_{l,\eps}^+) -\frac12 D\Phi^{red}_\eps(z_{l,\eps}^+)[z_{l,\eps}^+] \Big)
%+\Big(\frac12-\frac1\theta \Big)
%\int_{H_\eps^c}\la_\eps\cdot Q_\eps(x)|\bar z_{l,\eps}^+|^\theta dx
\geq c'.
\end{\equ}
%\todo[inline]{I have two remarks here: 1. by $(f4)$ and the definition of
%the truncation, if the limit in \eqref{E2} goes to zero, then we can deduce that
%$\bar z_{l,\eps}\to0$ in $L^\theta$ which implies
%$\bar z_{l,\eps}^+\to0$ in $L^\theta$ and $\bar z_{l,\eps}\to0$ in $E$, a contradiction
%to the fact $\|\bar z_{l,\eps}\|\geq c$;
%2. the constant $c'$ depends only on the Sobolev embedding (irrelevant to other
%factors like $R_1,R_2,R_3$) because it comes from \eqref{E1}. For details, please
%see discussion file 1801220-3}
Next, let us distinguish two possible situations.

\medskip
\textit{$\bullet$ Case 1.} $\bar\la\cdot y_j \geq 0$ for all $j=1,\dots,l$.

Since $\cb_\eps(u_\eps^+)=0$, we have that
\[
0=\int_{H_\eps}\la_\eps\cdot Q_\eps(x)|u_\eps^+|^\theta dx
+\int_{H_\eps^c}\la_\eps\cdot Q_\eps(x)|u_\eps^+|^\theta dx.
\]
By virtue of $z_{l,\eps}^+|_{H_\eps}\to0$ in the $L^\theta$-norm and
$\bar\la\cdot y_j\geq0$ for all $j=1,\dots,l$, we know that
\[
\int_{H_\eps}\la_\eps\cdot Q_\eps(x)|u_\eps^+|^\theta dx\to
\sum_{j=1}^l \bar\la\cdot y_j \int_{\R^3}|u_j^+|^\theta dx \geq0,
\]
whereas $\la_\eps\cdot Q_\eps(x)\geq \frac12\de>0$ in $H_\eps^c$.
Thus we have
\[
\bar\la\cdot y_j=0, \quad \text{for all } j=1,\dots,l,
\]
and so
\[
\frac\de2\int_{H_\eps^c}|u_\eps^+|^\theta dx\leq
\int_{H_\eps^c}\la_\eps\cdot Q_\eps(x)|u_\eps^+|^\theta dx\to0,\quad
\text{as } \eps\to0.
\]
We also deduce from \eqref{E0} that
\[
\int_{H_\eps^c}\la_\eps\cdot Q_\eps(x)|z_{l,\eps}^+|^\theta dx \to0,\quad
\text{as } \eps\to0.
\]
With all those information in hand, by Proposition \ref{h-vr converges},
we can estimate the energy $\Phi^{red}_\eps(u_\eps^+)$ as
\[
\Phi^{red}_\eps(u_\eps^+)=\Phi^{red}_\eps(z_{l,\eps}^+)+\sum_{j=1}^l \mst^{red}_{y_j}(u_j^+)+o_\eps(1).
\]
Moreover, we have that
\[
D\Phi^{red}_\eps(u_\eps^+)[u_\eps^+]=D\Phi^{red}_\eps(z_{l,\eps}^+)[z_{l,\eps}^+]+\sum_{j=1}^l D\mst^{red}_{y_j}(u_j^+)[u_j^+]+o_\eps(1).
\]
Since $\bar\la\cdot y_j = 0$ for all $j=1,\dots,l$, we have $u_j^+$'s
are critical points of $\mst^{red}_{y_j}$.
And so, we get the estimate
\[
\liminf_{\eps\to0} b_\eps=\liminf_{\eps\to0} \Phi^{red}_\eps(u_\eps^+)
=\liminf_{\eps\to0} \Big(\Phi^{red}_\eps(z_{l,\eps}^+) -\frac12D\Phi^{red}_\eps(z_{l,\eps}^+)[z_{l,\eps}^+] \Big)
+\sum_{j=1}^l \mst^{red}_{y_j}(u_j^+).
\]
Recall that we have denoted
$\bar z_{l,\eps}=z_{l,\eps}^++h_\eps(z_{l,\eps}^+)$, hence, by \eqref{E2},
we have
\[
\liminf_{\eps\to0} b_\eps\geq c'+\sum_{j=1}^l \mst^{red}_{y_j}(u_j^+).
\]

Since, by Lemma \ref{mod}, we have that $\mst^{red}_{y_j}(w)\geq J^{red}_{V(y_j)}(w)$, $\forall w\in E^+$,
for all $j=1,\dots,l$, we can infer that
\[
\mst^{red}_{y_j}(u_j^+)\geq \ga(J_{V(y_j)}), \quad j=1,\dots,l.
\]
And therefore
\[
\liminf_{\eps\to0} b_\eps\geq l\cdot\min_{j=1,\dots,l}\ga(J_{V(y_j)}) +c'.
\]
Remark that $y_j\in B_2=B(0,2R_1)$, by shrinking $R_1$ if necessary, we can conclude from the continuity of the map $\nu\mapsto \ga(J_\nu)$ that
\[
\big|\ga(J_{V(y_j)})-\ga(J_{\nu_0})\big|<\frac12 c' \quad \text{for all } j=1,\dots,l,
\]
and then we obtain
\[
\liminf_{\eps\to0} b_\eps\geq \ga(J_{\nu_0})+\frac12 c'>\ga(J_{\nu_0})
\]
which contradicts to Proposition \ref{gavr up estimate} and Lemma \ref{bvr is well defined}.

\medskip
\textit{$\bullet$ Case 2.} There exists $\{j_1,\dots,j_k\}\subset\{1,\dots,l\}$ such that
$\bar\la\cdot y_{j_m}<0$ for $m=1,\dots,k$.

In this case, similar as that in Case 1, we can apply Proposition \ref{h-vr converges}
to obtain
\[
\Phi^{red}_\eps(u_\eps^+)=\Big( \Phi^{red}_\eps(z_{l,\eps}^+) -\frac12D\Phi^{red}_\eps(z_{l,\eps}^+)[z_{l,\eps}^+] \Big)
+\sum_{j=1}^l \Big( \mst^{red}_{y_j}(u_j^+)-\frac12 D\mst^{red}_{y_j}(u_j^+)[u_j^+] \Big) +o_\eps(1).
\]
By the definition of $G(x,s)$, we have $\mst^{red}_{y_j}(u_j^+)-\frac12 D\mst^{red}_{y_j}(u_j^+)[u_j^+]\geq0$
for all $j=1,\dots,l$. Then, we conclude that
\begin{eqnarray}\label{E3}
\Phi^{red}_\eps(u_\eps^+)&\geq&\Big( \Phi^{red}_\eps(z_{l,\eps}^+) -\frac12D\Phi^{red}_\eps(z_{l,\eps}^+)[z_{l,\eps}^+] \Big)   \nonumber \\
& &+\sum_{m=1}^k\Big( \mst^{red}_{y_{j_m}}(u_{j_m}^+)-
\frac12 D\mst^{red}_{y_{j_m}}(u_{j_m}^+)[u_{j_m}^+] \Big) +o_\eps(1).
\end{eqnarray}

To evaluate the above inequality,
let us denote $\cm^+(\mst_{y_{j_m}})=\big\{ w\in E^+\setminus\{0\}:\, D\mst^{red}_{y_{j_m}}(w)[w]=0 \big\}$,
for $m=1,\dots,k$, and $t_m>0$ be the unique point such that
$t_m u_{j_m}^+\in\cm^+(\mst_{y_{j_m}})$. Observe that
$\bar\la\cdot y_{j_m}<0$, by Step 2 and Step 3, we get that
\[
D\mst^{red}_{y_{j_m}}(u_{j_m}^+)[u_{j_m}^+]-\bar\la\cdot y_{j_m}\int_{\R^3}|u_{j_m}^+|^\theta dx =0,
\]
and hence we have $t_m<1$. Observe that, by
applying Lemma \ref{key lemma}, we have
\[
\mst^{red}_{y_{j_m}}(u_{j_m}^+)-\frac12 D\mst^{red}_{y_{j_m}}(u_{j_m}^+)[u_{j_m}^+]
>\mst^{red}_{y_{j_m}}(t_mu_{j_m}^+)-\frac12 D\mst^{red}_{y_{j_m}}(t_mu_{j_m}^+)[t_mu_{j_m}^+].
\]
Then, it follows from
$t_mu_{j_m}^+\in\cm^+(\mst_{y_{j_m}})$ that
\[
\mst^{red}_{y_{j_m}}(u_{j_m}^+)-\frac12 D\mst^{red}_{y_{j_m}}(u_{j_m}^+)[u_{j_m}^+]
>\ga(J_{V(y_{j_m})}), \quad \text{for all } m=1,\dots,k.
\]
Finally, by \eqref{E2} and \eqref{E3}, we obtain the inequality
\[
\liminf_{\eps\to0}b_\eps=\liminf_{\eps\to0}\Phi^{red}_\eps(u_\eps^+)
\geq k\cdot\min_{m=1,\dots,k}\ga(J_{V(y_{j_m})})+c'.
\]
And therefore, as in Case 1, we conclude easily a contradiction.

\medskip
\textbf{Step 5.} Complete description of $u_\eps$ as $\eps\to0$.

As was argued in the previous steps, we know that there exists $l\in\N$
and, for any $j=1,\dots,l$, $y_\eps^j\in H_\eps$, $y_j\in B_2$ and
$u_j\in E\setminus\{0\}$ such that
\[
\aligned
&|y_\eps^j-y_\eps^{j'}|\to\infty, \quad \text{if } j\neq j',\\
&\eps y_\eps^j\to y_j,  \\
&\Big\|u_\eps^+-\sum_{j=1}^l u_j^+(\cdot-y_\eps^j)\Big\|\to0,\\
&D\mst^{red}_{y_j}(u_j^+)-\bar\la\cdot y_j\big(|u_j^+|^{\theta-2}u_j^+\big)^+=0.
\endaligned
\]
Observe that there strictly holds
\[
\mst^{red}_{y_{j}}(u_{j}^+)-\frac12 D\mst^{red}_{y_{j}}(u_{j}^+)[u_{j}^+]
>\ga(J_{V(y_{j})})
\]
provided that $\bar\la\cdot y_j<0$. Moreover, Lemma \ref{critical attained}
implies that $\ga(J_{V(y_j)})\geq \ga(J_{\nu_0})-\si$ for any $y_j\in B_2$, where
$\si>0$ can be taken arbitrary small by appropriately shrinking $R_1$. Therefore,
by Proposition \ref{gavr up estimate} and Lemma \ref{bvr is well defined},
we conclude that $l=1$ and $\bar\la\cdot y_1= 0$.
And thus we have
$\|u_\eps-u_1(\cdot-y_\eps^1)\|\to0$ as $\eps\to0$
 which complete the proof.
\end{proof}

\begin{Cor}
$y_1\in X^\bot$ and $\liminf_{\eps\to0}b_\eps\geq \ga(J_{\nu_0})$.
\end{Cor}
\begin{proof}
Since $\cb_\eps(u_\eps^+)=0$, by Proposition \ref{u-vr converges}, we get
\[
\aligned
0&=\int_{\R^3}Q_\eps(x)|u_\eps^+(x)|^\theta dx\\
&=\int_{\R^3}P_X(\zeta(\eps x +\eps y_\eps^1))
|u_\eps^+(x+y_\eps^1)|^\theta dx
 \to P_X(y_1)\int_{\R^3}|u_1^+|^\theta dx.
\endaligned
\]
Then $y_1\in X^\bot$, and we soon conclude
\[
\liminf_{\eps\to0}b_\eps\geq \ga(J_{V(y_1)})\geq\ga(J_{\nu_0}).
\]
\end{proof}

This finishes the proof of Proposition~\ref{key prop}.

\section{Profile of the solutions}\label{sec:Profile}

In this section, let us study the asymptotic behavior of the solution
$z_\eps$ obtained in
Theorem~\ref{existence part}. We will show that $z_\eps$ is actually a solution
of the original problem \eqref{Dirac0}, and consequently, we can complete the proof
of Theorem \ref{main result}.

Let us recall that $z_\eps$ is the critical point of $\Phi_\eps$ at level $\ga_\eps$, that is,
\begin{\equ}\label{profile1}
-i\al\cdot\nabla z_\eps + a\be z_\eps + V_\eps(x)z_\eps = g_\eps(x,|z_\eps|)z_\eps.
\end{\equ}
Moreover, Proposition \ref{gavr limit} implies that $\Phi_\eps(z_\eps)\to \ga(J_{\nu_0})$ as $\eps\to0$.

In what follows, we will give the asymptotic behavior of $z_\eps$ as $\eps\to0$.

\begin{Prop}\label{asymptotic prop}
Given a sequence $\eps_j\to0$, up to a subsequence, there exists $\{y_{\eps_j}\}\subset\R^3$
such that
\[
\eps_j y_{\eps_j}\to0, \quad \|z_{\eps_j}-Z(\cdot-y_{\eps_j})\|\to0,
\]
where $Z\in\msl_{\nu_0}$ $($see \eqref{notations}$)$.
\end{Prop}

\begin{proof}
For the sake of clarity, let us write $\eps=\eps_j$.
Our argument here has been used already in the previous section, so we will be
sketchy. First of all, analogous to Proposition \ref{u-vr converges},
we can conclude that: there exist
 $\bar y_\eps^1\in \R^3$, $\bar y_1\in B_2$ and $z_1\in E\setminus\{0\}$
with
\[
-i\al\cdot\nabla z_1+a\be z_1 + V(\bar y_1)z_1 = g(\bar y_1,|z_1|)z_1,
\]
such that
\[
\eps \bar y_\eps^1\to \bar y_1, \quad
\|z_\eps-z_1(\cdot-y_\eps^1)\|\to0 \quad \text{as } \eps\to0.
\]
So, the only thing that need to be proved is that $\bar y_1=0$.

By regularity arguments, $\{z_\eps\}\subset \cap_{q\geq2}W^{1,q}(\R^3,\C^4)$.
For arbitrary $\xi\in\R^3$,  multiplying \eqref{profile1} by $\pa_\xi z_\eps$ and
integrating, we get
\begin{\equ}\label{XXX}
-\frac\eps2\int_{\R^3} \pa_\xi V(\eps x) |z_\eps|^2dx+ \eps\int_{\R^3}
\big( F(|z_\eps|)-\tilde F(|z_\eps|) \big)
\pa_\xi \chi(\eps x)dx=0.
\end{\equ}
And if $\chi$ is $C^1$ around $\bar y_1$, we shall divide by $\eps$ and pass to the
limit to obtain
\begin{\equ}\label{chi c1}
-\frac{\pa_\xi V(\bar y_1)}2\int_{\R^3} |z_1|^2dx+ \pa_\xi \chi(\bar y_1)\int_{\R^3}
\big( F(|z_1|)-\tilde F(|z_1|) \big) dx=0.
\end{\equ}
At this point, similar as that in \cite{DPR}, we consider three different cases.

\medskip
\textit{$\bullet$ Case 1.} $\bar y_1\in B_1$.

By \eqref{chi c1}, we get that $\pa_\xi V(\bar y_1)=0$. Since $\xi\in\R^3$ is arbitrary,
$\bar y_1$ is a critical point of $V$ in $B_1$, and therefore $\bar y_1=0$.

\medskip
\textit{$\bullet$ Case 2.} $\bar y_1\in B_2\setminus \ov{B_1}$.

In this case, let us first fix $\xi=\frac1{|\bar y_1|}\bar y_1$.
By the definition of $\chi$ (see \eqref{chi}), we have that
$\pa_\xi \chi(\bar y_1)=-1/R_1$.

Now, using $(f3)$ and the fact $\tilde F(s)\leq \frac{\de_0}2 s^2$, it follow easily that there exists
a constant $c>0$ (which is independent of the choice of $\de_0$) such that
\[
\int_{\R^3} F(|z_1|)dx\geq c,
\]
and so by the boundedness of $z_1\in E$ (see an argument of Lemma \ref{boundedness}) we get
\[
c'\int_{\R^3}|z_1|^2dx \leq \int_{\R^3} \big( F(|z_1|)-\tilde F(|z_1|)\big)dx.
\]
Thus, it suffices to take $R_1$ smaller, if necessary,
to get a contradiction with \eqref{chi c1}.

\medskip
\textit{$\bullet$ Case 3.} $\bar y_1\in \pa B_2$.

In this case, observe that $\chi(\bar y_1)=1$, and so $z_1$ is a solution of
\[
-i\al\cdot\nabla z_1+a\be z_1 + V(\bar y_1)z_1 = f(|z_1|)z_1.
\]
Since $J^{red}_{V(\bar y_1)}(z_1^+)=J_{V(\bar y_1)}(z_1)=\ga(J_{\nu_0})$,
Lemma \ref{critical attained} implies that $V(\bar y_1)=\nu_0$. Then,
by \eqref{pa V}, there exists $\tau\in\R^3$ tangent to $\pa B_1$ at
$\bar y_1$ such that $\pa_\tau V(\bar y_1)\neq0$.

Remark that $\chi$ is not $C^1$ on $\pa B_1$,  let us go back to
consider \eqref{XXX}.
Take $\xi=\tau$ and $r<R_1$, we can estimate by the dominated convergence theorem
and the strong convergence of $z_\eps(\cdot+\bar y_\eps^1)$ that
\[
\aligned
&\Big| \int_{\R^3}\pa_\tau \chi(\eps x)\big[ F(|z_\eps|)-\tilde F(|z_\eps|)
 \big] dx\Big| \\
 \leq&\, \frac1{R_1}\int_{B(0,r/\sqrt{\eps})}\Big[ \frac{|x\cdot\tau|}{|x+\bar y_\eps^1|}
 + \frac{|\bar y_\eps^1\cdot\tau|}{|x+\bar y_\eps^1|} \Big]\big[ F(|z_\eps(x+\bar y_\eps^1)|)
 -\tilde F(|z_\eps(x+\bar y_\eps^1)|) \big] dx \\
 &\, +\frac1{R_1}\int_{\R^3\setminus B(0,r/\sqrt{\eps})}
 \frac{|(x+\bar y_\eps^1)\cdot\tau|}{|x+\bar y_\eps^1|}
\big[ F(|z_\eps(x+\bar y_\eps^1)|)
 -\tilde F(|z_\eps(x+\bar y_\eps^1)|) \big] dx \to0.
\endaligned
\]
Dividing by $\eps$ and passing to the limit in \eqref{XXX}, we can conclude
\[
\frac12\pa_\tau V(\bar y_1)\int_{\R^3} |z_1|^2dx=0,
\]
a contradiction.
\end{proof}

\begin{proof}[Complete proof of Theorem \ref{main result}]
It suffices to show that $|z_\eps(x)|\to0$ uniformly in $\R^3\setminus B_1^\eps$ as $\eps\to0$.
In fact, from the regularity argument in \cite[Lemma 3.19]{Ding-Wei-Xu-JMP},
we have that there exists $C>0$ (independent of $\eps$) such that
$|z_\eps|_{\infty}\leq C$. Then we can use elliptic esitmate
to get
\[
|z_\eps(x)|\leq C_0\int_{B(x,1)}|z_\eps(y)|dy
\]
with $C_0>0$ independent of both $\eps$ and $x\in\R^3$. And thus, by
Proposition \ref{asymptotic prop}, we have that
for any $x\in\R^3\setminus B_1^\eps$,
\[
\aligned
|z_\eps(x)|&\leq C_0\bigg(\int_{B(x,1)}|z_\eps|^2\bigg)^{1/2}\\
&\leq C_0\bigg(\int_{\R^3}\big|z_\eps-Z(\cdot - y_\eps)\big|^2\bigg)^{1/2}
+C_0\bigg(\int_{B(x,1)}\big|Z(\cdot - y_\eps)\big|^2\bigg)^{1/2} \to0,
\endaligned
\]
as $\eps\to0$. Finally, by the decay estimates obtained in \cite[Lemma 4.2]{Ding-Xu-ARMA},
it is standard to prove that there exists $C,c>0$ independent of $\eps$ such that
\[
|z_\eps(x)|\leq C \exp\big( -c|x-y_\eps| \big).
\]
This concludes the whole proof.
\end{proof}

\appendix
\section{Appendix}

\noindent
Here we sketch the proof of Proposition \ref{h-vr converges}. Firstly for later use let us point out that,
under the assumptions of Proposition \ref{h-vr converges},
$V(\eps\cdot+y_\eps)\to V(y)$ in $L_{loc}^\infty(\R^N)$ as $\eps\to0$. Now,
denote $V_\eps^0(x)=V(\eps x+y_\eps)-V(y)$, we
soon have
\begin{\equ}\label{R3}
\Phi_{\eps,y_\eps}(u)=\mst_y(u)+\frac12\int_{\R^3} V^0_\eps(x)|u|^2 dx
-\int_{\R^3} \big( G(\eps x+y_\eps,|u|)-G(y,|u|) \big) dx
\end{\equ}
for all $u\in E$. We also remark that,
for arbitrary $w\in E^+$ and $v\in E^-$,
by setting $\tilde v = v - h_{\eps,y_\eps}(w)$ and $\ell(t)=\Phi_{\eps,y_\eps}\big(w+h_{\eps,y_\eps}(w)+t\tilde v\big)$,
one has $\ell(1)=\Phi_{\eps,y_\eps}(w+v)$, $\ell(0)=\Phi_{\eps,y_\eps}\big(w+h_{\eps,y_\eps}(w)\big)$ and
$\ell'(0)=0$. Hence we deduce
$\ell(1)-\ell(0)=\int_0^1(1-s)\ell''(s)ds$. And
consequently, we have
\begin{\equ}\label{identity1}
\aligned
&\int_0^1(1-s)\Psi_{\eps,y_\eps}''\big(w+h_{\eps,y_\eps}(w)+s\tilde v\big)[\tilde v,\tilde v] \,ds\\
&+\frac12\|\tilde v\|^2+\frac12\int_{\R^N} V(\eps x+y_\eps)|\tilde v|^2dx=
\Phi_{\eps,y_\eps}\big(w+h_{\eps,y_\eps}(w)\big)-\Phi_{\eps,y_\eps}(z+v) ,
\endaligned
\end{\equ}
where, for notation convenience, we denote
$\Psi_{\eps,y}(u)\equiv\int_{\R^3}G(\eps x+y,|u|)dx$ for $u\in E$ and $y\in\R^3$.

Observe that assertion $(1)$ follows directly from \cite[Lemma 4.3]{Ding-Xu-Trans} and
that assertion $(3)$ can be viewed as an immediate corollary of assertion $(2)$. Hence,
to complete the proof, it suffices to show that, as $\eps\to0$,
\begin{\equ}\label{aim}
\left\{\aligned
&y_\eps\to y \text{ in } \R^3 \\
&w_\eps\weakto w \text{ in } E^+
\endaligned \right. %\ \text{ as } \eps\to0
\, \Longrightarrow  \,
\|h_{\eps,y_\eps}(w_\eps)-h_{\eps,y_\eps}(w_\eps-w)-h_y(w)\|=o_\eps(1).
\end{\equ}
To this end, we first claim that
\begin{equation}\label{BL1}
\begin{aligned}
  &\text{$y_\eps\to y$ in $\R^3$ and $u_\eps\weakto u$ in $E$ as $\eps\to0$}\\
  &\hspace{1cm} \Longrightarrow
    \Phi_{\eps,y_\eps}(u_\eps)-\Phi_{\eps,y_\eps}(u_\eps-u)-\Phi_{\eps,y_\eps}(u)=o_\eps(1) \quad
    \text{as } \eps\to0.
\end{aligned}
\end{equation}
\noindent
This can be proved similarly as \eqref{claim1} in Proposition \ref{PS condition}, therefore we omit the details. We only point out here that, for the nonlinear part, it suffices to check
\[
  \int_{\R^3}\big(G^1(\eps x+y_\eps,|u_\eps|)-G^1(\eps x+y_\eps,|u_\eps-u|)-G^1(\eps x+y_\eps,|u|)\big)dx=o_\eps(1)
\]
where $G^1(x,s)=G(x,s)-\frac\ka 3\chi(x) s^3$. Since $G^1$ is subcritical, the proof follows from a standard argument in \cite[Lemma 7.10]{Ding2}.

As a direct consequence of \eqref{BL1}, we soon conclude that
\begin{equation}\label{BL2}
  \text{For any sequence $w_\eps\weakto 0 $ in $E^+$, we have that $h_{\eps,y_\eps}(w_\eps)\weakto0$ in $E^-$.}
\end{equation}
Indeed, notice that $h_{\eps,y_\eps}(w_\eps)$ is bounded (see Theorem~\ref{thm:red-couple}),
we may assume up to a subsequence that
$h_{\eps,y_\eps}(w_\eps)\weakto u_0\in E^-$. Then
$u_\eps\equiv w_\eps+h_{\eps,y_\eps}(w_\eps)\weakto u_0$.
Now, remark that $\Psi_{\eps,y_\eps}\geq0$,
we conclude from \eqref{BL1} that
\[
\aligned
\frac{a-|V|_\infty}{2a}\|u_0\|^2\leq&\, -\Phi_{\eps,y_\eps}(u_0)
=  \Phi_{\eps,y_\eps}( u_\eps-u_0 )
-\Phi_{\eps,y_\eps}( u_\eps )
+o_\eps(1)
\leq o_\eps(1)
\endaligned
\]
as $\eps\to0$. And hence $u_0=0$.

\medskip

Now we are ready to show \eqref{aim}. Let $w_\eps\weakto w$ in $E^+$. We
may assume $h_{\eps,y_\eps}(w_\eps)\weakto v$ in $E^-$. By  \eqref{BL2},
there holds $h_{\eps,y_\eps}(w_\eps-w)\weakto0$. Using \eqref{BL1} and assertion $(1)$ (i.e. the
fact that $h_{\eps,y_\eps}(w)\to h_y(w)$ as $\eps\to0$),
we conclude that
\[
\aligned
\Phi_{\eps,y_\eps}\big(w_\eps+h_{\eps,y_\eps}(w_\eps)\big)&=\Phi_{\eps,y_\eps}(w+v)
+\Phi_{\eps,y_\eps}\big(w_\eps-w+h_{\eps,y_\eps}(w_\eps)-v\big)+o_\eps(1)  \\
&\leq\Phi_{\eps,y_\eps}\big(w+h_{\eps,y_\eps}(w)\big)
+\Phi_{\eps,y_\eps}\big(w_\eps-w+h_{\eps,y_\eps}(w_\eps-w)\big)+o_\eps(1)  \\
&=\Phi_{\eps,y_\eps}\big(w+h_y(w)\big)
+\Phi_{\eps,y_\eps}\big(w_\eps-w+h_{\eps,y_\eps}(w_\eps-w)\big)+o_\eps(1) \\
&=\Phi_{\eps,y_\eps}\big( w_\eps +h_{\eps,y_\eps}(w_\eps-w) +
h_y(w)\big)+o_\eps(1)
\endaligned
\]
as $\eps\to0$. Now use \eqref{identity1}, we can deduce that
\[
\frac{a-|V|_\infty}{2a}\|h_{\eps,y_\eps}(w_\eps)-h_{\eps,y_\eps}(w_\eps-w)-h_y(w)\|^2
\leq  o_\eps(1)
\]
and hence \eqref{aim} is proved.

%\medskip
%\noindent
%\textit{A.2. Final remark}
%
%Here we present

%\bigskip
%\footnotesize
%\noindent\textit{Acknowledgments.}
%This work was supported by the National Natural Science Foundation of China (NSFC11601370).

\vspace{2mm}
{\sc Thomas Bartsch\\
 Mathematisches Institut, Universit\"at Giessen\\
 35392, Giessen, Germany}\\
 Thomas.Bartsch@math.uni-giessen.de\\

{\sc Tian Xu\\
 Center for Applied Mathematics, Tianjin University\\
300072, Tianjin, China}\\
 xutian@amss.ac.cn
\end{document}